\documentclass[reqno]{amsart}
\usepackage{graphicx} % Required for inserting images
\usepackage{amsmath}%
\usepackage{amsfonts}%
\usepackage{amssymb}%
\usepackage{graphicx}
\usepackage{mathrsfs}
\usepackage{hyperref}
\usepackage{fullpage}
\usepackage{cite}
\usepackage{mathrsfs}
\usepackage{amsmath}
\usepackage{amssymb}
\usepackage{amsthm}
\usepackage{amsfonts}
\usepackage{amssymb}
\usepackage{amsmath}
\usepackage{enumitem}
\usepackage{mathtools}
\usepackage{mathrsfs}
\usepackage{color}
\usepackage{graphicx}

\allowdisplaybreaks
\def\0{\emptyset}

\def\0{\emptyset}

\newtheorem{theorem}{Theorem}
\newtheorem{lemma}[theorem]{Lemma}

\newtheorem{claim}[theorem]{Claim}

\newtheorem{conjecture}[theorem]{Conjecture}
\newtheorem{definition}[theorem]{Definition}

\newtheorem{problem}[theorem]{Problem}
\newtheorem{remark}[theorem]{Remark}

\usepackage{graphicx}
\usepackage{amsmath}
\usepackage{color}

\begin{document}

\begin{center}
{\LARGE A Quadratic Vertex Threshold for Isolated Cliques in the Minimum Degree Kruskal-Katona Problem  for 3-Uniform Hypergraphs}%
\end{center}

\vspace{0.3em}

\begin{center}
{\large Haorui Liu\,\textsuperscript{a}, \ Mei Lu\,\textsuperscript{a,\,1}, \ Yi Zhang\,\textsuperscript{b,c,\,$\ast$,\,2}}
\end{center}

\vspace{0.6em}

% ===== 单位 =====
\begin{center}
\begin{minipage}{0.92\textwidth}
\small
\begin{tabular}{@{}r@{\ }p{0.95\linewidth}@{}}
\textsuperscript{a} & Department of Mathematical Sciences, Tsinghua University, Beijing, 100084, P.R.\ China \\[0.25em]
\textsuperscript{b} & School of Mathematical Sciences, Beijing University of Posts and Telecommunications, Beijing, 100876, P.R.\ China \\[0.25em]
\textsuperscript{c} & Key Laboratory of Mathematics and Information Networks (BUPT), Ministry of Education, Beijing, 100876, P.R.\ China
\end{tabular}
\end{minipage}
\end{center}

% ===== 第一页底部脚注 =====
\renewcommand{\thefootnote}{\fnsymbol{footnote}}
\setcounter{footnote}{1}
\footnotetext[1]{Corresponding author: Yi Zhang.\\
\emph{E-mail addresses:} \texttt{lhr22@mails.tsinghua.edu.cn} (H.\ Liu), \texttt{lumei@tsinghua.edu.cn} (M.\ Lu), \texttt{shouwangmm@sina.com} (Y.\ Zhang).}
\renewcommand{\thefootnote}{\arabic{footnote}}
\setcounter{footnote}{0}
\footnotetext[1]{Mei Lu was supported by National Natural Science Foundation of China under Grant No.~12171272 and Beijing Natural Science Foundation 1262010.}
\footnotetext[2]{Yi Zhang is supported by Fundamental Research Funds for the Central Universities and Innovation Foundation of BUPT for Youth under Grant No.~2023RC49 and National Natural Science Foundation of China under Grant No.~11901048 and 12071002.}

\vspace{1.2em}

% ===== 摘要（手动排版，不要用 \begin{abstract}） =====
\begin{center}
\begin{minipage}{0.92\textwidth}
\small
\noindent\textbf{Abstract.} Given a set $X$ and an integer $t$, let $\mathcal{F}$ be a family of $k$-subsets of $X$.
The Kruskal-Katona theorem states that if $|\mathcal{F}|\geq \binom{t}{k}$, then
$|\partial_{k-1}\mathcal{F}|\geq\binom{t}{k-1}$. The minimum degree version of this problem
asks: if $\delta(\mathcal{F})\geq \binom{t}{k-1}$, how small can $|\partial_{k-1}\mathcal{F}|$
be? In this article, for the case $k=3$, we prove that, for every sufficiently large integer \(t\),   every extremal hypergraph for
this problem contains an isolated copy of $K_{t+1}^3$ whenever $|X| \geq ct^2 + o(t^2)$, with
the constant $c = 1 + \sqrt{928/33}$.
Our proof uses a graph transformation that
regularizes the neighborhood structure of extremal graphs, reducing
the problem to a counting argument on the neighbors of a
disjoint clique family. 
This gives a quadratic-order
threshold for the every-extremal version of the problem, compared with the
cubic-order threshold of F\"{u}redi and Zhao
[SIAM J.\ Discrete Math.\ 36(4), 2022]. 
\end{minipage}
\end{center}

\vspace{0.8em}

\medskip
\noindent% ===== 关键词 & MSC =====
\begin{center}
\begin{minipage}{0.92\textwidth}

\end{minipage}
\end{center}

\medskip

\section{Introduction}

Let $X$ be a finite set and let $\binom{X}{k}$ denote the collection of all $k$-subsets of $X$.
For a family $\mathcal{F} \subseteq \binom{X}{k}$, the $\ell$-shadow of $\mathcal{F}$, denoted by $\partial_\ell \mathcal{F}$,
is defined as the family of all $\ell$-subsets of $X$ that are contained in at least one member of $\mathcal{F}$. For any real number $x \ge k$, we define the generalized binomial coefficient as $\binom{x}{k} = \frac{x(x-1)\cdots(x-k+1)}{k!}$.

\begin{theorem}[Lovász Version of the Kruskal-Katona Theorem\cite{Lovász}]\label{lovas} If $\mathcal{F}$ is a $k$-uniform family with $|\mathcal{F}| \ge \binom{x}{k}$ for some real $x \ge k$, then for any $\ell \le k$, the size of its $\ell$-shadow satisfies
    \[ |\partial_\ell \mathcal{F}| \ge \binom{x}{\ell}, \]
    and equality occurs if and only if $x$ is an integer and $\mathcal{F}$ is the family of all $k$-subsets of some $x$-element set.
\end{theorem}

A family $\mathcal{F}$ of $k$-subsets of $X$ can also be regarded as a $k$-uniform hypergraph with vertex set $X$. A $k$-uniform hypergraph (briefly, $k$-graph) $H$ is a pair $(V, E)$, where $V = V (H)$ is a finite set of vertices and $E = E(H)$ is a family of $k$-element subsets of $V$. Given a set $S \subseteq V$, the degree $\deg_H(S)$ of $S$ is the number of edges of $H$ that contain $S$. We write $\deg(S)$ when $H$ is clear from the context. If $S = \{u\}$, then we write $\deg(u)$ instead of $\deg(\{u\})$. Let $\delta(H) = \min\{ \deg(u) : u \in V (H)\}$.  We define $N_H(u) = \{v \in V(H) : v \neq u \text{ and } \exists\, e \in E(H) \text{ with } \{u,v\} \subseteq e\}$, and for any set $S \subseteq V$, define $N_H(S) = \bigcup_{u \in S} N_H(u)$.  The \textit{induced subgraph} of $H$ on $S$, denoted by $H[S]$, is the $k$-graph with vertex set $S$ and edge set
$E(H[S]) = \{e \in E(H) : e \subseteq S\}$. A \textit{clique} of $H$ is a set $S \subseteq V(H)$ such that every $k$-element subset of $S$ is an edge of $H$, i.e., $\binom{S}{k} \subseteq E(H)$.

Recently,
 extremal problems under a minimum or maximum degree condition have received considerable attention. Frankl\cite{Fra2} studied the Erd{\H{o}}s--Ko--Rado theorem under the condition of a maximum degree. Frankl and Tokushige\cite{Fra4}, Huang and Zhao\cite{Huang1}, Kupavskii\cite{Kup} and Huang and Zhang \cite{Huang2} studied the minimum degree version of the Erd{\H{o}}s--Ko--Rado theorem. Frankl, Han, Huang and Zhao\cite{Fra3} studied the minimum degree version of the Hilton-Milner theorem. In addition, there has been extensive research on the degree version of the Erd{\H{o}}s matching conjecture\cite{Han, Kha1, Mar, Kuhn1, Pik, Rod2, Rod3, TrZh12, TrZh13,wang, Yi4, zhang, zhang2} and Tur\'{a}n problems\cite{Lo, Lo1, Si, Mu}. Chase \cite{Ch} studied the Kruskal-Katona theorem under a condition of maximum degree.

 The minimum degree version of the Kruskal--Katona theorem was initiated by F\"uredi and Zhao \cite{Fu2}, who posed the following two equivalent problems.

\begin{problem}[F\"uredi--Zhao \cite{Fu2}]\label{prob:FZ-shadow}
Given an integer $k \geq 3$ and $t\geq k-1$, let $X$ be a set of $n$ vertices and $\mathcal{F}$ be a family of $k$-subsets of $X$ with $\delta(\mathcal{F}) \geq \binom{t}{k-1}$. Determine the minimum possible size of the $(k-1)$-shadow $|\partial_{k-1} \mathcal{F}|$.
\end{problem}

\begin{problem}[F\"uredi--Zhao \cite{Fu2}]\label{13}
Given an integer $k \geq 3$ and $t\geq k-1$, determine the minimum number of edges in a $(k-1)$-graph $G$ on $n$ vertices such that every vertex is contained in at least $\binom{t}{k-1}$ copies of $K^{k-1}_k$ (the complete $(k-1)$-graph on  $k$ vertices).
\end{problem}

The equivalence follows from the
correspondence between a $k$-uniform family $\mathcal{F}$ and the
$(k-1)$-graph $G$ whose edges are the $(k-1)$-subsets in
$\partial_{k-1}\mathcal{F}$: the minimum degree condition on
$\mathcal{F}$ translates directly into a clique-degree condition on
$G$. For a detailed proof of this equivalence, we refer the interested
reader to~\cite{Fu2}.

In this paper, we focus on the $k=3$ case, which
corresponds to the following problem: determine
the minimum number of edges in a graph $G$ on $n$ vertices such that
every vertex is contained in at least $\binom{t}{2}$ triangles.

This is closely related to the classical Rademacher--Tur\'{a}n problem, initiated by Rademacher (unpublished) and Erd\H{o}s \cite{Erdos}, which seeks the minimum number of triangles in a graph of given order and size. A natural variant is to consider the minimum number of triangles containing a fixed vertex.
F\"uredi and Zhao \cite{Fu2} studied this problem and determined the extremal graphs for $t+1 \leq n \leq 2(t+1)$. As the precise statement is rather elaborate, we refer the interested reader to \cite{Fu2} for details. For large $n$, they further proved the following structural result.

\begin{theorem}\cite{Fu2}\label{pro}
    When $n>\frac{1}{4}(t+1)^2(t+2)$ and $t \ge 2$ is an integer, every extremal graph for Problem \ref{13} with $k=3$  contains an isolated copy of $K_{t+1}$.
\end{theorem}

Note that by Theorem~\ref{pro}, if $n > \frac{1}{4}(t+1)^2(t+2)$, we can delete the vertices in the isolated copy of $K_{t+1}$ from $V(G)$ and study the remaining part of $G$. Iterating this procedure reduces the number of vertices until $n \leq \frac{1}{4}(t+1)^2(t+2)$. On the other hand, F\"{u}redi and Zhao \cite{Fu2}  determine the extremal graphs for $n \leq 2(t+1)$. There remains a gap between $2(t+1)$ and $\frac{1}{4}(t+1)^2(t+2)$. Thus, closing the gap would completely solve Problem~\ref{13} for $k=3$.  
In this article, for sufficiently large \(t\), we replace this cubic threshold
by a quadratic one for the same every-extremal conclusion.

For our proof, we  use the following result of F\"uredi and Zhao \cite{Fu2}, which provides an upper bound on the total
``excess'' degree sum in any extremal graph. In particular, it implies
that all but at most $O(t^2)$ vertices have degree exactly $t$, so
each such vertex belongs to a unique copy of $K_{t+1}$.

\begin{lemma}[F\"uredi--Zhao \cite{Fu2}]\label{d}
Let $G$ be an extremal graph for Problem~\ref{13} with $k=3$. Then
\[
  \sum_{v \in V(G)} \bigl(\deg(v) - t\bigr)
  \;\leq\; \tfrac{1}{4}(t+1)^2.
\]
\end{lemma}

The main result of this paper is the following improvement of
Theorem~\ref{pro}, reducing the threshold from $O(t^3)$ to $O(t^2)$.

\begin{theorem}\label{3}
Let $c=1+\sqrt{928/33}.$
For every sufficiently large integer \(t\), if
$n\ge ct^2+o(t^2),$ then every extremal graph for Problem~\ref{13} with $k = 3$ contains an isolated copy of $K_{t+1}$.
\end{theorem}

Since Problems~\ref{prob:FZ-shadow} and~\ref{13} are equivalent,
Theorem~\ref{3} may be restated in terms of 3-uniform hypergraphs as
follows.

\begin{theorem}\label{3-uniform-version}
Let $c=1+\sqrt{928/33}.$
For every sufficiently large integer \(t\), if
$|X|=n\ge ct^2+o(t^2),$
then  every extremal \(3\)-uniform hypergraph  for
Problem~\ref{prob:FZ-shadow} with $k=3$   contains an isolated copy of \(K^3_{t+1}\).
\end{theorem}

\begin{remark}
The assumption that \(t\) is sufficiently large is used only to
avoid tracking inessential lower-order terms and boundary constants throughout
the proof. In fact, the argument works once \(t\) is larger than some fixed
absolute constant; no asymptotically growing lower bound on \(t\) is required.
\end{remark}

\section{A useful construction}

In this section, we introduce a graph transformation, denoted by $\mathcal{G}$, which allows us to regularize the structure of an extremal graph while maintaining its edge minimality and minimum triangle-degree condition. The transformation relies on
a specific reassignment of neighborhoods.

\begin{definition}\label{11}
     Let $G = (V, E)$ be an extremal graph for Problem \ref{13} with $k=3$. Suppose $A_1, A_2 \subseteq V$ (not necessarily distinct) both induce a $K_{t+1}$ in $G$, with $A_1 = \{v_1, v_2, \dots, v_{t+1}\}$. Let $B=N (A_1 \cup A_2)\setminus (A_1 \cup A_2)$. We define the transformed graph $G' = \mathcal{G}(G, A_1, A_2)$ through the following operations:

\begin{itemize}
    \item[(i)]  Set $V(G') = V$.
    \item[(ii)]  The induced subgraphs $G'[A_1]$ and $G'[A_2]$ remain $K_{t+1}$, while all edges between $A_1 \setminus A_2$ and $A_2 \setminus A_1$ are removed.% ensuring Property (3).
    \item[(iii)] Set $G'[V \setminus (A_1 \cup A_2)] = G[V \setminus (A_1 \cup A_2)]$.
    \item[(iv)]  For each vertex $u \in B$, we set $N_{G'}(u) \cap (A_2 \setminus A_1) = \emptyset$. Its neighborhood within $A_1$ is defined as $N_{G'}(u) \cap A_1 = \{v_1, v_2, \dots, v_{m_u}\}$, where $m_u = \min\{ |N_G(u) \cap (A_1 \cup A_2)|, t+1 \}$.
\end{itemize}
\end{definition}

\begin{lemma}\label{G}
     The following properties hold for $G' = \mathcal{G}(G, A_1, A_2)$.

\begin{enumerate}
    \item[(1)] \textbf{Edge Non-increase:} $|E(G')| \le |E(G)|$.
    \item[(2)] \textbf{Edge Confinement:} Both $A_1$ and $A_2$ induce a  $K_{t+1}$ in $G'$. Furthermore, for every  $e \in E(G')$, if $e \cap (A_2 \setminus A_1) \neq \emptyset$, then $e \subseteq A_2$.
    \item[(3)] \textbf{Clique Shifting:} The construction \textbf{ensures} that for any $A_3 \in \binom{V(G)}{t+1}$ that induces a $K_{t+1}$ in $G$, its ``shifted'' counterpart
    \[ A_4 = (A_3 \setminus (A_1 \cup A_2)) \cup \{v_{1}, v_2, \dots, v_{m}\} \]
    also induces a $K_{t+1}$ in $G'$, where $m = |A_3 \cap (A_1 \cup A_2)|$.
    \item[(4)] \textbf{Maintenance of Triangle Condition:} Every vertex in $G'$ remains contained in at least $\binom{t}{2}$ triangles.
\end{enumerate}
\end{lemma}

\begin{proof}
Properties (1) and (2) follow directly from the construction. For Property (3), consider a set $A_3\in \binom{V(G)}{t+1}$ that induces a $K_{t+1}$ in $G$. The vertices of $A_3$ outside $A_1 \cup A_2$ and the edges between them remain unchanged by step (iii). For any vertex $u \in A_3 \setminus (A_1 \cup A_2)$, let $m=|N_G(u)\cap A_3 \cap (A_1 \cup A_2)|$. According to the reassignment in step (iv), these $m$ neighbors are replaced in $G'$ by the initial segment $\{v_1, \dots, v_m\} \subseteq A_1$. Since $A_1$ induces a $K_{t+1}$ in $G'$, all vertices in $\{v_1, \dots, v_m\}$ are automatically pairwise adjacent, and the adjusted adjacency rule ensures that $u$ is adjacent to each of them. Consequently, $A_4$ induces a $K_{t+1}$ in $G'$, satisfying Property (3).

It remains to  verify Property (4), namely that  the minimum triangle-degree condition is preserved. We partition the vertex set $V(G')$ into three cases. If $u \in A_1 \cup A_2$, then by the construction, $u$ belongs to a copy of $K_{t+1}$ in $G'$, which ensures that it is contained in at least $\binom{t}{2}$ triangles. If $u \notin A_1 \cup A_2 \cup B$, the edges incident to $u$ and its local structure are preserved from $G$, so the
triangle-degree of $u$ remains unchanged.

The remaining case is when $u \in B$. Let $t_i$ and $t'_i$ denote the number of triangles containing $u$ that have exactly $i$ vertices in $A_1 \cup A_2$ in $G$ and $G'$, respectively. Note that $t'_0 = t_0$. In  $G$, the triangle-degree of $u$ satisfies:
\[ \binom{t}{2} \le t_0 + t_1 + t_2 \le t_0 + \sum_{w \in B \cap N_G(u)} |N_G(w) \cap N_G(u) \cap (A_1 \cup A_2)| + \binom{|N_G(u) \cap (A_1 \cup A_2)|}{2}. \]

To evaluate the triangle-degree of $u$ in $G'$, we consider the value of $|N_{G'}(u)\cap A_1|$. If $|N_G(u) \cap (A_1 \cup A_2)| \ge t + 1$, then step (iv) ensures that $|N_{G'}(u)\cap A_1|=t+1$ and $u$ is adjacent to every vertex in $A_1$. In this case, $\{u\} \cup A_1$ induces a $K_{t+2}$ in $G'$, and $u$ is contained in at least $\binom{t+1}{2} \ge \binom{t}{2}$ triangles. Assume that $|N_G(u) \cap (A_1 \cup A_2)| \le t$. Then we have $|N_{G'}(u) \cap A_1| = |N_G(u) \cap (A_1 \cup A_2)|$. The reassignment in step (iv) ensures that for any $u, w \in B$, their neighborhoods restricted to $A_1$ are nested initial segments. This nesting maximizes the common neighborhood size such that
\[ |N_{G'}(w) \cap N_{G'}(u) \cap A_1| = \min\{|N_{G'}(w) \cap A_1|, |N_{G'}(u) \cap A_1|\}. \]
By the construction of $G'$, we have $|N_{G'}(w) \cap A_1|=\min\{|N_G(w) \cap (A_1 \cup A_2)|, t+1\}$, and we also know $|N_G(u) \cap (A_1 \cup A_2)| \le t$. Consequently,
\[ |N_{G'}(w) \cap N_{G'}(u) \cap A_1| = \min\{|N_G(w) \cap (A_1 \cup A_2)|, |N_G(u) \cap (A_1 \cup A_2)|\} \ge |N_G(w) \cap N_G(u) \cap (A_1 \cup A_2)|. \]
Furthermore, since $A_1$ induces a $K_{t+1}$ in $G'$, $t'_2 = \binom{|N_{G'}(u) \cap A_1|}{2}$. Combining these, we obtain:
\begin{align*}
t'_0 + t'_1 + t'_2 &= t_0 + \sum_{w \in B \cap N_{G}(u)} \min\{|N_{G'}(w) \cap A_1|, |N_{G'}(u) \cap A_1|\} + \binom{|N_{G'}(u) \cap A_1|}{2} \\
&\ge t_0 + \sum_{w \in B \cap N_G(u)} |N_G(w) \cap N_G(u) \cap (A_1 \cup A_2)| + \binom{|N_G(u) \cap (A_1 \cup A_2)|}{2} \\
&\ge t_0 + t_1 + t_2 \ge \binom{t}{2}.
\end{align*}
This completes the proof of Property (4).
\end{proof}

 \section{Proof of Theorem \ref{3}}

We shall use the following special case of Lemma~23 in~\cite{liu}. 
In the notation of~\cite{liu}, the case \(k=3\) and \(\ell=2\) is precisely
Problem~\ref{13} with \(k=3\).

\begin{lemma}\cite[Lemma~23]{liu}\label{lem23}
Let \(t\ge 2\), and let \(G\) be an extremal graph for Problem~\ref{13}
with \(k=3\) and
\[
|V(G)|>\frac14(t+1)^2+3t+1.
\]
Let \(\mathcal A\) be the family of all vertex sets that induce a copy of
\(K_{t+1}\) in \(G\). Then there is at most one unordered pair
\(\{A_1,A_2\}\subseteq \mathcal A\) such that
\(A_1\ne A_2\) and \(A_1\cap A_2\ne\emptyset\).
\end{lemma}

Indeed, this follows by taking \((k,\ell)=(3,2)\) in Lemma~23 of~\cite{liu};
the second alternative in that lemma cannot occur when \(\ell=2\).

Suppose, for the sake of contradiction, that $G$ is an extremal graph on $n \ge (c+o(1))t^2$ vertices, where $c = 1 + \sqrt{928/33}$, containing no isolated copy of $K_{t+1}$. Let $\mathcal{A}$ be the family of all vertex sets that induce a $K_{t+1}$ in $G$. We partition $\mathcal{A}$ into $\mathcal{A}_1$ and $\mathcal{A}_2$ by setting
\[
\mathcal{A}_1 := \left\{ A \in \mathcal{A} : \sum_{v \in A} (d(v) - t) < \frac{\sqrt{33}}{2\sqrt{928}}(t+1) \right\}
\]
and $\mathcal{A}_2 := \mathcal{A} \setminus \mathcal{A}_1$.
Let $V_1 = \{v \in V(G) : \deg_G(v) = t\}$ and $V_2 = V(G) \setminus V_1$. By the Kruskal-Katona theorem (Theorem \ref{lovas}), every vertex in $V_1$ is contained in a unique $K_{t+1}$ in $G$.

\begin{claim}\label{gr}
    For any $A_1,A_2\in \mathcal{A}$, there is no edge in $G$ between $A_1\setminus A_2$ and $A_2\setminus A_1$.
\end{claim}

\begin{proof}
    Suppose that there are $A_1,A_2\in \mathcal{A}$ and an edge $e\in E(G)$ that has a nonempty intersection with both $A_1\setminus A_2$ and $A_2\setminus A_1$. From the construction of $\mathcal{G}$, it follows that $|E(\mathcal{G}(G,A_1,A_2))|<|E(G)|$. By Lemma \ref{G}, every vertex in $\mathcal{G}(G,A_1,A_2)$ is contained in at least $\binom{t}{2}$ triangles. 
    Thus \(\mathcal G(G,A_1,A_2)\) is an admissible
graph for Problem~\ref{13}  with \(k=3\), but it has fewer edges than \(G\), contradicting
the extremality of \(G\).
\end{proof}

\begin{claim} \label{claim30}
There are at least
\[
\left( \frac{\sqrt{928}}{2\sqrt{33}} + \frac{3}{4} \right)(t+1) + o(t)
\]
pairwise disjoint cliques in $\mathcal{A}_1$.
\end{claim}

{\bf Note.} To make the calculation more intuitive, we do not simplify $\frac{\sqrt{928}}{2\sqrt{33}}$.

\begin{proof}
By Lemma \ref{lem23}, there exists a subcollection of pairwise disjoint cliques $\mathcal{A}' \subseteq \mathcal{A}$ such that $|\mathcal{A}'| \ge |\mathcal{A}| - 1$.
Since every vertex in \(V_1\) belongs to a unique copy of \(K_{t+1}\) in
\(\mathcal A\), and since \(G\) contains no isolated copy of \(K_{t+1}\),
each member of \(\mathcal A\) contains at most \(t\) vertices of \(V_1\).
Hence $|\mathcal{A}| \ge (n - |V_2|)/t$.

Combined with the bound $|V_2| \le \frac{1}{4}(t+1)^2$ from Lemma \ref{d}, a straightforward calculation yields
\begin{equation} \label{eq:A-total}
|\mathcal{A}'| \ge |\mathcal{A}| - 1 \ge \frac{n - (t+1)^2/4}{t} - 1 = \left( \frac{\sqrt{928}}{\sqrt{33}} + \frac{3}{4} \right)(t+1) + o(t).
\end{equation}

By the definition of $\mathcal{A}_2$, every $A \in \mathcal{A}_2 \cap \mathcal{A}'$ satisfies $\sum_{v \in A} (d(v) - t) \ge \frac{\sqrt{33}}{2\sqrt{928}}(t+1)$. Since the cliques in $\mathcal{A}'$ are pairwise disjoint, we have
\[
|\mathcal{A}_2 \cap \mathcal{A}'| \cdot \frac{\sqrt{33}}{2\sqrt{928}}(t+1) \le \sum_{A \in \mathcal{A}_2 \cap \mathcal{A}'} \sum_{v \in A} (d(v) - t) \le \sum_{v \in V(G)} (d(v) - t).
\]
By Lemma \ref{d}, the rightmost sum is at most $\frac{1}{4}(t+1)^2$. It follows that
\[
|\mathcal{A}_2 \cap \mathcal{A}'| \le \frac{2\sqrt{928}}{\sqrt{33}(t+1)} \cdot \frac{(t+1)^2}{4} = \frac{\sqrt{928}}{2\sqrt{33}}(t+1).
\]
Note that the number of cliques in $\mathcal{A}_1 \cap \mathcal{A}'$ is at least $|\mathcal{A}'| - |\mathcal{A}_2 \cap \mathcal{A}'|$. Subtracting the above estimate from \eqref{eq:A-total}, we obtain the desired lower bound.
\end{proof}

Let $\mathcal{D}$ be the collection of pairwise disjoint cliques in $\mathcal{A}_1$ provided by Claim \ref{claim30}. By Claim \ref{gr},
    for any $A_1,A_2\in \mathcal{D}$, there is no edge in $G$ between $A_1$ and $A_2$. Since $G$  contains no isolated copy of $K_{t+1}$, $N(D)\setminus D\not=\emptyset$ for any $D\in \mathcal{D}$.
We construct a sequence of vertex sets $U_1, U_2, \dots$ and sub-families $ \mathcal{D}_1 \supseteq \mathcal{D}_2 \supseteq \dots$ by initializing $i = 1$ and $\mathcal{D}_1 = \mathcal{D}$, and proceeding as follows:

\begin{enumerate}
    \item[(1)] Let $V(\mathcal{D}_i) = \bigcup_{D \in \mathcal{D}_i} D$ and  $B_i = N(V(\mathcal{D}_i)) \setminus V(\mathcal{D}_i)$.

    \item[(2)] For each $u \in B_i$, let $f_i(u) = |N(u) \cap V(\mathcal{D}_i)|$. Choose a vertex $u_i \in B_i$ that maximizes $f_i(u_i)$.

    \item[(3)] Let $U_i = \{ v \in B_i : N(v) \cap V(\mathcal{D}_i) = N(u_i) \cap V(\mathcal{D}_i) \}$. We denote $a_i = |U_i|$ and write $U_i = \{u_i^1, \dots, u_i^{a_i}\}$ with $u_i^1 = u_i$.

    \item[(4)] Define $\mathcal{D}_{i+1} = \{ D \in \mathcal{D}_i : D \cap N(u_i) = \emptyset \}$. If $\mathcal{D}_{i+1} \neq \emptyset$, increment $i$ and repeat from Step (1); otherwise, let $\xi=i$ be the total number of iterations performed, and the process terminates.
\end{enumerate}

\begin{figure}[!htbp]
\begin{center}
\includegraphics[scale=0.2]{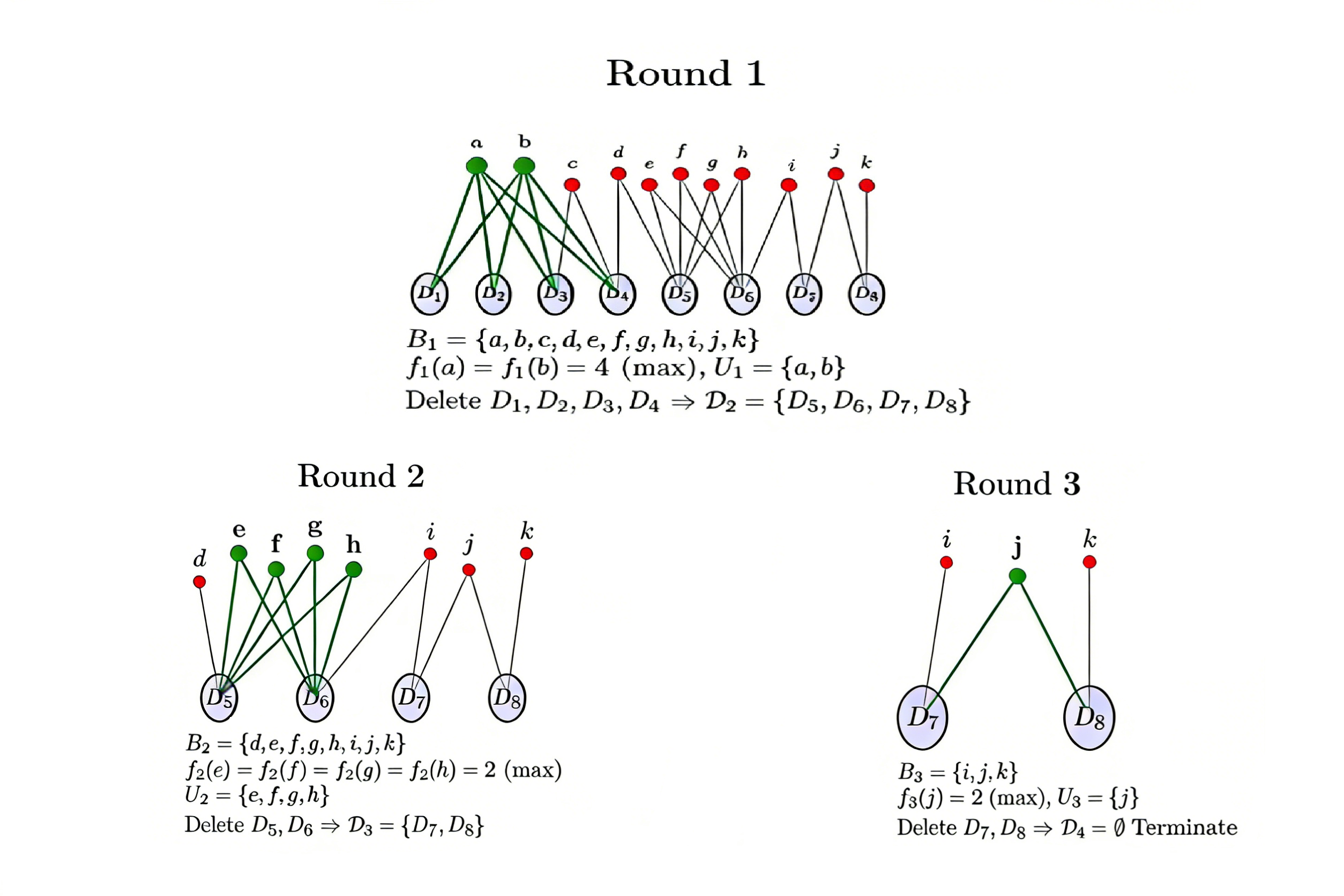}\\
\caption{An example for the construction of $U_i$ and $\mathcal{D}_i$.}\label{Figure1}
\end{center}
\end{figure}

We begin by collecting several technical lemmas needed for the proof of Theorem \ref{3}. To maintain the flow of the main argument, their proofs are deferred to  the end of this section.

Set $\alpha = 4/7$ and $\Lambda=f_1(u_1)$. For each index $1\le i\le\xi$, let $b_i$ be the largest integer satisfying
\begin{equation} \label{def:bi}
f_{i+b_i-1}(u_{i+b_i-1}) \ge
\begin{cases}
     \frac{\lceil 2\alpha f_i(u_i) \rceil}{2}+1  & \text{if } f_i(u_i) \ge 3, \\[8pt]
    f_i(u_i) & \text{if } f_i(u_i) \in \{1, 2\}.
\end{cases}
\end{equation}
The following two lemmas provide an upper bound on the length of these intervals.

% --- Lemma 1 ---
\begin{lemma} \label{b}
%Set $\alpha = 4/7$.
For any $1\le i\le \xi $, if $f_i(u_i) \ge 2$,
then $b_i \le 6\alpha f_i(u_i) + 7$. Furthermore, if $f_i(u_i) \ge 8$, this bound improves to $b_i \le 4\alpha f_i(u_i) + 5$.
\end{lemma}

% --- Lemma 2 ---
\begin{lemma} \label{d1}
For any $1\le i\le \xi $, if $f_i(u_i) = 1$, then $b_i \le \frac{3}{4}t$.
\end{lemma}

The following lemma establishes a structural constraint on the vertices in $B_1$, showing that a vertex with a large neighborhood in $V(\mathcal{D})$ can only be adjacent to a few cliques in the family.

\begin{lemma}\label{id}
 Let $\theta = \frac{\sqrt{33}}{2\sqrt{928}}(t+1)$. For any vertex $u \in B_1$, let $k(u)=|\{D\in \mathcal{D}:N_G(u) \cap D \neq \emptyset\}|$.
    \begin{enumerate}
        \item[(i)] If $|N_G(u) \cap V(\mathcal{D})| \ge \sqrt{2\theta} + 3$, then $k(u) \le 3$.
        \item[(ii)] If $|N_G(u) \cap V(\mathcal{D})| \ge \sqrt{3\theta} + 5$, then $k(u) \le 2$.
\end{enumerate}
\end{lemma}

%Now we  begin to prove Theorem \ref{3}.

{\noindent \bf Proof of Theorem \ref{3}:}
We now establish an upper bound on the number of $(t+1)$-cliques in $\mathcal{D}$ by analyzing the decay of the maximum neighborhood sizes. Let $\omega$ be the smallest index such that $f_\omega(u_\omega) < \sqrt{2\theta} + 3$, where $\theta$ is given in Lemma \ref{id}. If no such index exists, we let $\omega = \infty$. We define an index sequence $i_1, i_2, \dots$ by setting $i_1 = 1$ and, for $j \ge 1$,
\[
i_{j+1} =
\begin{cases}
\omega & \text{if } i_j < \omega \le i_j + b_{i_j}, \\
i_j + b_{i_j} & \text{otherwise}.
\end{cases}
\]
The sequence $\{i_1, i_2, \dots\}$ ends when $i_j>\xi$, where $\xi$ is the number of iterations. Let $j_0$ denote the length of the sequence $\{i_1, i_2, \dots\}$, that is, $\{i_1, i_2, \dots\}=\{i_j\}_{j=1}^{j_0}$.%Let $j_0$ denote the length of the sequence $\{i_j\}_{j=1}^{j_0}$.

If $\omega \neq \infty$, let $q$ be the unique index satisfying $i_q = \omega$;  if $\omega = \infty$, set $q=j_0$. For any $j \neq q$, the definition of $b_{i_j}$ ensures the recursive bound $f_{i_{j}}(u_{i_{j}}) < \alpha f_{i_{j-1}}(u_{i_{j-1}}) + 1$, where $\alpha = 4/7$. This recurrence implies a geometric decay towards the fixed point $1/(1-\alpha) = 7/3$. Specifically, for $j < q$, a standard induction yields
\begin{equation} \label{eq:decay-pre}
f_{i_j}(u_{i_j}) < \alpha^{j-k}\left(f_{i_k}(u_{i_k}) - \frac{7}{3}\right) + \frac{7}{3} \le \alpha^{j-1}\Lambda + 3,
\end{equation}
where the final inequality follows by taking $k = 1$ and noting that $f_{i_1}(u_{i_1}) \leq \Lambda$.
Analogously, for the phase starting at $\omega$, for any $j \ge q$, the neighborhood sizes satisfy
\begin{equation} \label{eq:decay-post}
f_{i_j}(u_{i_j}) < \alpha^{j-q}\left(\sqrt{2\theta} + 3\right) + 3.
\end{equation}

 For each $1\le i \le\xi$, let $k'(u_i)$ denote the number of cliques in the current family $\mathcal{D}_i$ that have a nonempty intersection with $N_G(u_i)$, and recall that $k(u)$ is defined as the number of cliques $D\in \mathcal{D}$ such that $N_G(u) \cap D \neq \emptyset$. Note that $k'(u_i)\leq k(u_i)$ for any $1\le i\le \xi$.

 To estimate $|\mathcal{D}|$, we partition the index sequence $\{i_j\}_{j=1}^{j_0} $ into five intervals, determined by the thresholds established in Lemmas \ref{b}, \ref{d1}, and \ref{id}.
\begin{enumerate}
    \item Let $p$ be the maximal index such that $f_{i_p}(u_{i_p}) \ge \alpha^{-1}(\sqrt{3\theta} + 5)$; if no such index exists, set $p=0$. For all $1 \le j \leq p$ and $i_j\leq i < i_{j+1}$, the definition of $b_{i_j}$ ensures that $f_{i}(u_{i}) \ge \sqrt{3\theta} + 5$. Thus, Lemma \ref{id}(ii) implies $k'(u_{i}) \le 2$,  while Lemma \ref{b} yields $b_{i_j} \le 4\alpha f_{i_j}(u_{i_j}) + 5$.

    \item For $p+1 \le j \leq q-1$ and $i_j\leq i < i_{j+1}$, where $i_q = \omega$ is the previously defined threshold, Lemma \ref{id}(i) yields $k'(u_{i}) \le 3$, while the bound $b_{i_j} \le 4\alpha f_{i_j}(u_{i_j}) + 5$ remains valid.

    \item Let $r$ be the minimal index such that $f_{i_r}(u_{i_r}) < 8$; if no such index exists, set $r=j_0$. For $q \le j \le r-1$ and $i_j\leq i < i_{j+1}$, we use the trivial bound $k'(u_{i}) \le f_{i}(u_{i}) \le f_{i_j}(u_{i_j})$, while $b_{i_j} \le 4\alpha f_{i_j}(u_{i_j}) + 5$ still holds.

    \item Let $s$ be the minimal index such that $f_{i_s}(u_{i_s}) = 1$ and it follows from the definition of $b_i$ and $j_0$ that $s=j_0-1$ if such index exists; if no such index exists, set $s=j_0$. For $r \le j \le s-1$ and $i_j\leq i < i_{j+1}$, we again have $k'(u_{i}) \le f_{i}(u_{i}) \le f_{i_j}(u_{i_j})$, but the step size bound increases to $b_{i_j} \le 6\alpha f_{i_j}(u_{i_j}) + 7$ by Lemma \ref{b}.

    \item For $j=s$ and $i_j \le i\le \xi$, since $f_{i}(u_{i}) = 1$, each vertex $u_{i}$ is contained in exactly one clique, and $b_{i_j} \le \frac{3}{4}t$ by Lemma \ref{d1}.
\end{enumerate}

Combining the bounds from all five cases above,  we obtain an upper bound on the total number of $(t+1)$-cliques in $\mathcal{D}$:
\begin{equation*}
    \sum_{i=1}^\xi k'(u_i) \le \sum_{i=1}^{i_{p+1}-1} 2 + \sum_{i=i_{p+1}}^{i_{q}-1} 3 + \sum_{i=i_{q}}^{i_r-1} f_{i}(u_{i}) + \sum_{i=i_r}^{i_s-1} f_{i}(u_{i}) + \frac{3}{4}t.
\end{equation*}
Recalling that each jump $j$ in our index sequence $i_j$ covers at most $b_{i_j}$ individual iterations, we substitute the bounds on $b_{i_j}$. Then we have
\begin{align*}
\sum_{i=1}^\xi k'(u_i) &\le 2 \sum_{j=1}^{p} (4\alpha f_{i_j}(u_{i_j}) + 5) + 3 \sum_{j=p+1}^{q-1} (4\alpha f_{i_j}(u_{i_j}) + 5) \\
    &\quad + \sum_{j=q}^{r-1} f_{i_j}(u_{i_j})(4\alpha f_{i_j}(u_{i_j}) + 5) + \sum_{j=r}^{s-1} f_{i_j}(u_{i_j})(6\alpha f_{i_j}(u_{i_j}) + 7) + \frac{3}{4}t.
\end{align*}

Applying the upper bounds $f_{i_j}(u_{i_j}) \le \alpha^{j-1}\Lambda + 3$ for $1 \le j \le p$, $f_{i_j}(u_{i_j}) \le \alpha^{j-p-2}(\sqrt{3\theta} + 5) + 3$ for $p+1 \le j \le q$, and $f_{i_j}(u_{i_j}) \le \alpha^{j-q}(\sqrt{2\theta} + 3) + 3$ for $q \le j \le r-1$, along with the specific terminal values $f_{i_r}(u_{i_r}) \le 7, f_{i_{r+1}}(u_{i_{r+1}}) \le 4, f_{i_{r+2}}(u_{i_{r+2}}) \le 3$, and $f_{i_{r+3}}(u_{i_{r+3}}) \le 2$ (which follow from the recurrence $f_{i_j}(u_{i_j}) < \alpha f_{i_{j-1}}(u_{i_{j-1}}) + 1$ with $\alpha = 4/7$), we obtain
\begin{align*}
\sum_{i \ge 1} k'(u_i) &\le 2 \left( \sum_{j=1}^{p} 4\alpha^j \Lambda + \Theta(p) \right) + 3 \left( \sum_{j=p+1}^{q-1} 4\alpha^{j-p-1} \sqrt{3\theta} + \Theta(q-p) \right) \\
    &\quad + \sum_{j=q}^{r-1} \left( 4\alpha \cdot (\alpha^{j-q}\sqrt{2\theta})^2 + \Theta(\sqrt{\theta}) \right) + \sum_{k=0}^3 (6\alpha f_{i_{r+k}} + 7)f_{i_{r+k}} + \frac{3}{4}t \\
    &= 8\Lambda \left( \frac{\alpha}{1-\alpha} \right) + 8\theta \left( \frac{\alpha}{1-\alpha^2} \right) + \frac{3}{4}t + \Theta(\sqrt{\theta}).
\end{align*}
Evaluating the coefficients with $\alpha = 4/7$, a straightforward calculation yields the final upper bound:
\begin{equation} \label{eq:final_clique_bound}
\sum_{i \ge 1} k'(u_i) \le \frac{32}{3}\Lambda + \frac{224}{33}\theta + \frac{3}{4}t + \Theta(\sqrt{\theta}).
\end{equation}

Lemma \ref{id} implies that $\Lambda \le 2\theta$. Indeed, if there is a vertex $u \in B_1$ with $f_1(u)=\Lambda > 2\theta$, then for sufficiently large $t$, the condition $\Lambda > \sqrt{3\theta} + 5$ would hold, and Lemma \ref{id}(ii) would yield $k(u) \le 2$. By the pigeonhole principle, $u$ would have more than $\Lambda/2 > \theta$ neighbors in some clique $D \in \mathcal{D}$, which contradicts the definition of $\mathcal{A}_1$.

Substituting $\Lambda \le 2\theta$ into our previous estimate, we have that
\begin{align*}
|\mathcal{D}| = \sum_{i \ge 1} k'(u_i)  &\le \left( \frac{64}{3} + \frac{224}{33} \right)\theta + \frac{3}{4}t + \Theta(\sqrt{\theta}) \\
&= \frac{928}{33}\theta + \frac{3}{4}t + \Theta(\sqrt{\theta}).
\end{align*}
On the other hand, Claim \ref{claim30} establishes the lower bound:
\begin{align*}
|\mathcal{D}| &\ge \frac{\sqrt{928}}{2\sqrt{33}}(t+1) + \frac{3}{4}t + o(t) \\
&= \frac{928}{33}\theta + \frac{3}{4}t + o(t).
\end{align*}

For sufficiently large $t$, we may choose the $o(t^2)$ term in Theorem \ref{3} to be $t^{11/6}$. It follows from Claim \ref{claim30} that $\mathcal{A}_1$ contains at least
\[
\left( \frac{\sqrt{928}}{2\sqrt{33}} + \frac{3}{4} \right)(t+1) + \Theta(t^{5/6})
\]
pairwise disjoint cliques. Since $\Theta(t^{5/6})$ dominates $\Theta(\sqrt{\theta})$, the upper bound of $|\mathcal{D}|$ is strictly less than this lower bound for large $t$. This contradiction completes the proof of Theorem \ref{3}.
\qed

\section{Proofs of Technical Lemmas}

In this section, we prove Lemmas \ref{b}, \ref{d1}, and \ref{id}. The proof of Lemma \ref{b}  is the most intricate. Throughout, we let  $\mathcal{D} = \{D_1, \ldots, D_d\}$  and  $V(\mathcal{D}) = \bigcup_{j=1}^d D_j$.
The cliques $D_1, D_2, \dots,D_d$ are pairwise vertex-disjoint; furthermore, Claim \ref{gr}  ensures that no two distinct cliques are joined by an edge. Let $D_1 = \{v_1, \ldots, v_{t+1}\}$.
Define $G_0 = G$ and iteratively construct  $G_j = \mathcal{G}(G_{j-1}, D_1, D_j)$ for $1 \leq j \leq d$. By Lemma \ref{G}, every vertex in $G_d$  is contained  in at least $\binom{t}{2}$ triangles.
Since $|E(G_d)|\le|E(G)|$, we must have  $|E(G_d)| = |E(G)|$.   We claim that $\Lambda=f_1(u_1) \leq t$. Suppose $\Lambda \geq t+1$. Then $u_1v_{t+1}\in E(G_d)$ which implies $\{u_1\} \cup D_1$ forms a $K_{t+2}$. Deleting $v_tv_{t+1}$ from $G_d$ preserves the triangle-degree condition of Problem \ref{13} with $k=3$ while reducing edges, a contradiction.

\subsection{Lemma \ref{b}}

\begingroup
\it
For any $1\le i\le \xi $, if $f_i(u_i) \ge 2$, we have $b_i \le 6\alpha f_i(u_i) + 7$. Moreover, if $f_i(u_i) \ge 8$, then $b_i \le 4\alpha f_i(u_i) + 5$.
\endgroup

\medskip
By the translation invariance of the iterative procedure, it suffices to establish the result for $i = 1$. Let $b = b_1$, $U = \bigcup_{i=1}^b U_i$ and recall that $\Lambda = f_1(u_1)$. To clarify the subsequent technical details, we first outline the main ideas of the proof.

\medskip
\noindent \textbf { Overview of the proof.}
Assume, for the sake of contradiction, that $b > 6\alpha\Lambda + 7$ (respectively, $b > 4\alpha\Lambda + 5$ if $\Lambda \ge 8$). To contradict the extremality of $G$, we construct a graph $J$ on the same vertex set $V(G)$ such that $|E(J)| < |E(G)|$, while ensuring that every vertex in $J$ is still contained in at least $\binom{t}{2}$ triangles.

The construction hinges on identifying a large independent set $X_0 \subseteq U$. We transform $X_0$ into a clique in $J$ and redistribute the adjacencies between $B_1\setminus X_0$ and $V(\mathcal{D}) \cup X_0$ to maintain the minimum triangle-degree condition. Crucially, the assumption on  $b > 6\alpha\Lambda + 7$ guarantees the existence of such an independent set $X_0$ via a greedy procedure. The reduction in the number of edges is achieved because the edges removed between $X_0$ and $V(\mathcal{D})$ more than offset the new edges created within $X_0$, thus yielding the desired contradiction.

\begin{proof}
For any $1 \leq i \leq b$, we have

\begin{equation}\label{equaas}
\Lambda = f_1(u_1) \geq f_1(u_i) \geq f_i(u_i) \geq f_i(u_b) \geq f_b(u_b)  \geq \lceil \alpha \Lambda \rceil,
\end{equation}
the final inequality follows from (\ref{def:bi}).

Let $Z =\{ u \in U : N(u) \cap U \subseteq U_i \text{ for the index } i \text{ such that } u \in U_i,1\le i\le b \}. $ We partition the family $\{U_1, \ldots, U_b\}$ into two sub-families: $\mathcal{U}_1 = \{U_i : U_i \cap Z = \emptyset\}$ and $\mathcal{U}_2 = \{U_i : U_i \cap Z \neq \emptyset\}$. Assume $|\mathcal{U}_1| = b'$ and then $|\mathcal{U}_2| = b-b'$.

A central component in the proof of Lemma \ref{b} is Claim \ref{m}, which provides an upper bound on the cross-adjacencies for vertices in the sub-collection $\mathcal{U}_1$. While by the definitions of $Z$ and $\mathcal{U}_2$, each set $U_i$ in $\mathcal{U}_2$ contains at least one vertex that is isolated from $U \setminus U_i$, the sets in $\mathcal{U}_1$ lack such representatives. Claim \ref{m} addresses this by showing that any $u \in U_i \in \mathcal{U}_1$ still has a very limited neighborhood in $U \setminus U_i$; this property is essential for constructing an independent set of a specific size in $U$, the existence of which is fundamental to the proof of Lemma \ref{b}. We establish Claim \ref{m} via the two auxiliary claims below. For the sake of flow, the technical details of their proofs may be skipped initially.

\begin{claim}\label{01}
    For any $1\le i\le \xi$, $u\in U_i$ and $u'\in B_1\backslash U_i$, we have $$\min\left\{-1+\sum\limits_{j=0}^d|N_{G}(u)\cap D_j|,\sum\limits_{j=0}^d|N_{G}(u')\cap D_j|\right\}\geq \sum\limits_{j=0}^d|N_{G}(u)\cap N_G(u')\cap D_j|.$$
\end{claim}

\begin{proof}
Suppose there is $1\le i\le \xi$, $u\in U_i$ and $u'\in B_1\backslash U_i$ such that
\[
\min\left\{ -1 + \sum_{j=0}^d |N_G(u) \cap D_j|, \sum_{j=0}^d |N_G(u') \cap D_j| \right\} + 1 \le \sum_{j=0}^d |N_G(u) \cap N_G(u') \cap D_j|.
\]
On the other hand, we have the following general chain of inequalities:
\begin{align}
\min\left\{ -1 + \sum_{j=0}^d |N_G(u) \cap D_j|, \sum_{j=0}^d |N_G(u') \cap D_j| \right\} + 1
&\ge \min\left\{ \sum_{j=0}^d |N_G(u) \cap D_j|, \sum_{j=0}^d |N_G(u') \cap D_j| \right\} \label{0000}\\
&\ge \sum_{j=0}^d \min\left\{ |N_G(u) \cap D_j|, |N_G(u') \cap D_j| \right\} \label{0001}\\
&\ge \sum_{j=0}^d |N_G(u) \cap N_G(u') \cap D_j|. \label{0002}
\end{align}
Comparing these bounds, it follows that all inequalities in the chain must hold as equalities. The equality in (\ref{0002}) implies that for each $0 \le j \le d$, we have either $N_G(u) \cap D_j \subseteq N_G(u') \cap D_j$ or $N_G(u) \cap D_j \supseteq N_G(u') \cap D_j$. Furthermore, for the equality in (\ref{0001}) to hold, we have that for each $0 \le j \le d$, either $N_G(u) \cap D_j \subseteq N_G(u') \cap D_j$ or $N_G(u) \cap D_j \supseteq N_G(u') \cap D_j$.

Combined with the equality in (\ref{0000}), we conclude that $N_G(u) \cap D_j \subseteq N_G(u') \cap D_j$ for all $0 \le j \le d$.
Since  $f_i(u) \ge f_i(u')$, this containment forces $N_G(u) \cap D_j = N_G(u') \cap D_j$ for all $D_j \in \mathcal{D}_i$. Consequently, $u' \in U_i$, yielding a contradiction.
\end{proof}

\begin{claim}\label{02}
Let $1 \le i \neq i' \le b$. For any vertices $u \in U_i$ and $u' \in U_{i'}$, we have
\[
\min\left\{ \sum_{j=0}^d |N_G(u) \cap D_j|, \sum_{j=0}^d |N_G(u') \cap D_j| \right\} \ge \sum_{j=0}^d |N_G(u) \cap N_G(u') \cap D_j| + \lceil \alpha \Lambda \rceil.
\]
\end{claim}

\begin{proof}
Without loss of generality, assume $i > i'$. For any $x \in \{u, u'\}$, let $N^*(x) = \bigcup_{j=0}^d (N_G(x) \cap D_j)$. Since the cliques $\{D_j\}_{j=0}^d$ are pairwise disjoint, we have  $|N^*(x)| = \sum_{j=0}^d |N_G(x) \cap D_j|$.

First, by restricting the neighborhood of $u$ to the cliques in $\mathcal{D}_i$, we observe that
\begin{equation*} \label{eq:bound_u}
|N^*(u) \setminus N^*(u')| \ge \left| \bigcup_{D \in \mathcal{D}_i} (N_G(u) \cap D) \setminus N^*(u') \right|.
\end{equation*}
By construction, the neighborhood $N^*(u')$ is disjoint from the cliques in $\mathcal{D}_i$. Combining   this with (\ref{equaas}), we obtain:
\begin{equation} \label{eq:bound_u1}
\left| \bigcup_{D \in \mathcal{D}_i} (N_G(u) \cap D) \setminus N^*(u') \right| = \sum_{D \in \mathcal{D}_i} |N_G(u) \cap D| \ge \lceil \alpha\Lambda \rceil.
\end{equation}
It follows that
\begin{equation} \label{eq:bound_u11}
|N^*(u) \setminus N^*(u')| \ge \lceil \alpha\Lambda \rceil.
\end{equation}

Next, to bound $|N^*(u') \setminus N^*(u)|$, we restrict the neighborhood difference to the cliques in $\mathcal{D}_{i'}$. We have
\begin{equation*}
|N^*(u') \setminus N^*(u)| \ge \left| \bigcup_{D \in \mathcal{D}_{i'}} (N_G(u') \cap D) \right| - \left| \bigcup_{D \in \mathcal{D}_{i'}} (N_G(u') \cap N_G(u) \cap D) \right|.
\end{equation*}
By the degree property of our partition,  $| \bigcup_{D \in \mathcal{D}_{i'}} (N_G(u') \cap D) | \ge | \bigcup_{D \in \mathcal{D}_{i'}} (N_G(u) \cap D) |$. Consequently,
\begin{equation*}
|N^*(u') \setminus N^*(u)| \ge \left| \bigcup_{D \in \mathcal{D}_{i'}} (N_G(u) \cap D) \right| - \left| \bigcup_{D \in \mathcal{D}_{i'}} (N_G(u) \cap N_G(u') \cap D) \right| \ge | \bigcup_{D \in \mathcal{D}_{i'}} (N_G(u) \cap D) \setminus N^*(u') |.
\end{equation*}
Since $\mathcal{D}_{i'} \supseteq \mathcal{D}_i$, this last expression is bounded below by the cardinality of the union restricted to $\mathcal{D}_i$. By  \eqref{eq:bound_u1}, we conclude that
\begin{equation} \label{eq:bound_u12}
|N^*(u') \setminus N^*(u)| \ge \lceil \alpha\Lambda \rceil.
\end{equation}

Finally, applying the identity $|A \cap B| = |A| - |A \setminus B|$ to both $u$ and $u'$, and combining (\ref{eq:bound_u11}) and (\ref{eq:bound_u12}) we obtain
\[
|N^*(u) \cap N^*(u')| \le \min\{ |N^*(u)|, |N^*(u')| \} - \lceil \alpha\Lambda \rceil.
\]
Substituting the summation forms for the neighborhood sizes completes the proof.
\end{proof}

\begin{claim}\label{m} Suppose $\Lambda \ge 2$. For any $U_i \in \mathcal{U}_1$, every vertex $u_i^j \in U_i$ ($1\le j\le a_i$) satisfies
\[
|N_G(u_i^j) \cap (U \setminus U_i)| \le \frac{\Lambda + a_i - 2}{\lceil \alpha \Lambda \rceil - 1}.
\]
\end{claim}

\begin{proof}

Following the construction of $G_d$, for each vertex $u \in B_1$, we have $N_{G_d}(u) \cap D_1 = \{v_1, v_2, \dots, v_{f_1(u)}\}$. Given $U_i \in \mathcal{U}_1$ and $u_i^j \in U_i$.
Let $G''$ be the graph obtained from $G_d$ by deleting the edge $e = u_i^j v_{f_1(u_i^j)}$. For any vertex $u \in V(G)$, let $t_k(u)$ and $t_k''(u)$ denote the number of triangles containing $u$ that intersect $V(\mathcal{D})$ in exactly $k$ vertices in $G$ and $G''$, respectively.

By the extremality of $G_d$, there must exist a vertex $u \in V(G'')$ that is contained in fewer than $\binom{t}{2}$ triangles. Since $G''$ is obtained by removing only the edge $e$, $u$ must be adjacent to $u_i^j$ in $G_d$.

Note that  $G''[V(G) \setminus V(\mathcal{D})] = G_d[V(G) \setminus V(\mathcal{D})] = G[V(G) \setminus V(\mathcal{D})]$. Consequently, every vertex in $V(G) \setminus (V(\mathcal{D}) \cup B_1)$ is contained in at least $\binom{t}{2}$ triangles in $G''$.  Furthermore, since each $D_i\in \mathcal{D}$ induces a $K_{t+1}$ in $G''$ for $1 \le i \le d$, every vertex in $V(\mathcal{D})$ is also contained in at least $\binom{t}{2}$ triangles.

It follows that $u$ must belong to $B_1$. Moreover, since the deleted edge $e$ is incident to $V(\mathcal{D})$, the number of triangles containing $u$ that do not intersect $V(\mathcal{D})$ remains unchanged, i.e., $t_0''(u) = t_0(u)$.

    \textbf{Case 1:} $u\in B_1\setminus U_i$.

Since $u \not=u_i^j $, $u$ is not incident to the deleted edge $e$. Consequently, the collection of triangles containing $u$ with two vertices in $V(\mathcal{D})$ is unaffected by the removal of $e$. It follows that

\[
t_2''(u) = \binom{|N_{G''}(u) \cap D_1|}{2} = \binom{\sum_{h=0}^d |N_G(u) \cap D_h|}{2} \ge t_2(u).
\]
Next, we consider triangles containing $u$ that intersect $V(\mathcal{D})$ in exactly one vertex. We have
\begin{align*}
t_1''(u) &= \sum_{u' \in B_1 \cap N_{G''}(u)} |N_{G''}(u') \cap N_{G''}(u) \cap D_1| \\
&= \sum_{u' \in B_1 \cap N_{G''}(u) \setminus \{u_i^j\}} |N_{G''}(u') \cap N_{G''}(u) \cap D_1| + |N_{G''}(u_i^j) \cap N_{G''}(u) \cap D_1|.
\end{align*}
By the construction of $G''$ and  $G_d$, this can be rewritten as
\begin{align*}
t_1''(u) &= \sum_{u' \in B_1 \cap N_G(u) \setminus \{u_i^j\}} \min\left\{ \sum_{h=0}^d |N_G(u') \cap D_h|, \sum_{h=0}^d |N_G(u) \cap D_h| \right\} \\
&\quad + \min\left\{ -1 + \sum_{h=0}^d |N_G(u_i^j) \cap D_h|, \sum_{h=0}^d |N_G(u) \cap D_h| \right\}.
\end{align*}
For $u' \in B_1 \setminus \{u_i^j\}$, the neighborhoods are unaffected by the deletion of the edge $e = u_i^j v_{f_1(u_i^j)}$.  Thus, the first term in the summation satisfies
\[\sum_{u' \in B_1 \cap N_G(u) \setminus \{u_i^j\}} \min\left\{ \sum_{h=0}^d |N_G(u') \cap D_h|, \sum_{h=0}^d |N_G(u) \cap D_h| \right\}
\ge \sum_{u' \in B_1 \cap N_G(u) \setminus \{u_i^j\}} \sum_{h=0}^d |N_G(u) \cap N_G(u') \cap D_h|.
\]
Applying Claim  \ref{01} to the second term, we obtain
\[
\min\left\{ -1 + \sum_{h=0}^d |N_G(u_i^j) \cap D_h|, \sum_{h=0}^d |N_G(u) \cap D_h| \right\} \ge \sum_{h=0}^d |N_G(u) \cap N_G(u_i^j) \cap D_h|.
\]

Combining these two results, we obtain
\begin{align*}
t_1''(u) &\ge \sum_{u' \in B_1 \cap N_G(u) \setminus \{u_i^j\}} \sum_{h=0}^d |N_G(u) \cap N_G(u') \cap D_h| + \sum_{h=0}^d |N_G(u) \cap N_G(u_i^j) \cap D_h| = t_1(u).
\end{align*}
Summing the contributions $t_2''(u), t_1''(u),$ and $t_0''(u)$, we find that $u$ is contained in at least $\binom{t}{2}$ triangles, yielding a contradiction.

    \textbf{Case 2:} $u\in U_i\backslash\{u_i^j\}$.

    Since $U_i \in \mathcal{U}_1$, there exists a vertex $u_{i'}^{j'} \in U \setminus U_i$ adjacent to $u$. Note that as $u \neq u_i^j$, its neighborhood in $V(\mathcal{D})$ remains unchanged in $G''$. We first observe that
\[
t_2''(u) = \binom{|N_{G''}(u) \cap D_1|}{2} = \binom{\sum_{h=0}^d |N_G(u) \cap D_h|}{2} \ge t_2(u).
\]

Next, we analyze $t_1''(u)$ by partitioning the sum over neighbors in $B_1 \cap N_{G''}(u)$. We have
\begin{align*}
t_1''(u) &= \sum_{u' \in B_1 \cap N_{G''}(u) \setminus \{u_i^j, u_{i'}^{j'}\}} |N_{G''}(u') \cap N_{G''}(u) \cap D_1| \\
&\quad + |N_{G''}(u_i^j) \cap N_{G''}(u) \cap D_1| + |N_{G''}(u_{i'}^{j'}) \cap N_{G''}(u) \cap D_1|.
\end{align*}
By the construction of $G''$, the neighborhood of $u_i^j$ in $D_1$ is reduced by exactly one vertex, while the neighborhood of $u$ and the neighborhoods of all other vertices in $B_1 \cap N_{G''}(u) \setminus \{u_i^j\}$ remain as in $G_d$. Furthermore,  by the construction of $G_d$,  this allows us to bound $t_1''(u)$ using the neighborhood sizes relative to $V(\mathcal{D})$ in  $G$:
\begin{align*}
t_1''(u) &\ge \sum_{u' \in B_1 \cap N_G(u) \setminus \{u_i^j, u_{i'}^{j'}\}} \min\left\{ \sum_{h=0}^d |N_G(u') \cap D_h|, \sum_{h=0}^d |N_G(u) \cap D_h| \right\} \\
&\quad + \min\left\{ -1 + \sum_{h=0}^d |N_G(u_i^j) \cap D_h|, \sum_{h=0}^d |N_G(u) \cap D_h| \right\} \\
&\quad + \min\left\{ \sum_{h=0}^d |N_G(u_{i'}^{j'}) \cap D_h|, \sum_{h=0}^d |N_G(u) \cap D_h| \right\}.
\end{align*}
We now apply the intersection bounds. For any $u' \in B_1 \cap N_G(u)$, the size of the intersection $|N_G(u) \cap N_G(u') \cap V(\mathcal{D})|$ is at most the minimum of their respective neighborhood sizes in $V(\mathcal{D})$. Crucially, since $u \in U_i$ and $u_{i'}^{j'} \in U \setminus U_i$ belong to distinct partition classes, by Claim \ref{02} for the term involving $u_{i'}^{j'}$ and substituting the intersection sums, we obtain:
\begin{align*}
t_1''(u) &\ge \sum_{u' \in B_1 \cap N_G(u) \setminus \{u_i^j, u_{i'}^{j'}\}} \sum_{h=0}^d |N_G(u) \cap N_G(u') \cap D_h| \\
&\quad + \left( \sum_{h=0}^d |N_G(u) \cap N_G(u_i^j) \cap D_h| - 1 \right) + \left( \sum_{h=0}^d |N_G(u) \cap N_G(u_{i'}^{j'}) \cap D_h| + \lceil \alpha\Lambda \rceil \right) \\
&= t_1(u) + \lceil \alpha\Lambda \rceil - 1 \ge t_1(u).
\end{align*}
Summing these contributions, we find $t_2''(u) + t_1''(u) + t_0''(u) \ge t_2(u) + t_1(u) + t_0(u) \ge \binom{t}{2}$, which contradicts the extremality of $G_d$.

    \textbf{Case 3:} $u=u_i^j$.

Let $N_0 = N_G(u) \cap (U \setminus U_i)$. In this case, $u$ is an endpoint of the deleted edge $e$, hence its neighborhood size in $V(\mathcal{D})$ is reduced by one. We first bound $t_2''(u)$:
\begin{align*}
t_2''(u) &= \binom{|N_{G''}(u) \cap V(\mathcal{D})|}{2} = \binom{|N_{G_d}(u) \cap V(\mathcal{D})| - 1}{2}\\&= \binom{|N_{G_d}(u) \cap V(\mathcal{D})|}{2}-(|N_{G_d}(u) \cap V(\mathcal{D})| - 1)  \ge t_2(u) - (\Lambda - 1).
\end{align*}

Next, we evaluate $t_1''(u)$ by partitioning the neighbors of $u$ in $B_1 \cap N_{G''}(u)$ into three disjoint sets: $B_1 \setminus (U_i \cup N_0)$, $U_i \setminus \{u\}$, and $N_0$.

The deletion of the edge $e$ reduces the neighborhood size of $u$ in $V(\mathcal{D})$ by one.
For $u' \in (B_1 \cap N_{G''}(u)) \setminus (U_i \cup N_0)$, the neighborhood of $u'$ in $V(\mathcal{D})$ is unchanged, and we have
\[
|N_{G''}(u') \cap N_{G''}(u) \cap V(\mathcal{D})| \ge \min\Big\{-1 + \sum_{j=0}^d |N_G(u) \cap D_j|,\, \sum_{j=0}^d |N_G(u') \cap D_j|\Big\}.
\]
By Claim \ref{01}, it follows that
\[ |N_{G''}(u') \cap N_{G''}(u) \cap V(\mathcal{D})| \ge\sum_{h=0}^d |N_G(u) \cap N_G(u') \cap D_h|. \]
For  $u' \in U_i \setminus \{u\}$, since $u'$ is not incident to $e$, the neighborhood of $u'$ in $V(\mathcal{D})$ is unchanged. The only change is that $u$ loses one neighbor in $V(\mathcal{D})$, so the intersection size satisfies
\[
|N_{G''}(u') \cap N_{G''}(u) \cap V(\mathcal{D})| \ge \sum_{h=0}^d |N_G(u) \cap N_G(u') \cap D_h| - 1.
\]
For $u' \in N_0$, since $u \in U_i$ and $u' \in U \setminus U_i$ belong to distinct partition classes, Claim~\ref{02} yields an additional margin of $\lceil \alpha\Lambda \rceil$. After accounting for the loss of one due to the edge deletion, the net increment is at least $\lceil \alpha\Lambda \rceil - 1$:
\[
|N_{G''}(u') \cap N_{G''}(u) \cap V(\mathcal{D})| \ge \sum_{h=0}^d |N_G(u) \cap N_G(u') \cap D_h| + \lceil \alpha\Lambda \rceil - 1.
\]

Combining these three contributions, we obtain

\begin{align*}
t_1''(u) &= \sum_{u' \in B_1 \cap N_G(u) } |N_{G''}(u') \cap N_{G''}(u) \cap D| \\
&\ge \sum_{u' \in B_1 \cap N_G(u) \setminus (U_i \cup N_0)} \sum_{h=0}^d |N_G(u) \cap N_G(u') \cap D_h| \\
&\quad + \sum_{u' \in U_i \setminus \{u\}} \left( \sum_{h=0}^d |N_G(u) \cap N_G(u') \cap D_h| - 1 \right) \\
&\quad + \sum_{u' \in N_0} \left( \sum_{h=0}^d |N_G(u) \cap N_G(u') \cap D_h| + \lceil \alpha\Lambda \rceil - 1 \right) \\
&= t_1(u) - (a_i - 1) + (\lceil \alpha\Lambda \rceil - 1)|N_0|.
\end{align*}

Summing the contributions, the total number of triangles containing $u$ in $G''$ satisfies:
\begin{align*}
t_2''(u) + t_1''(u) + t_0''(u) &\ge t_2(u) - (\Lambda - 1) + t_1(u) - (a_i - 1) + (\lceil \alpha\Lambda \rceil - 1)|N_0| + t_0(u) \\
&\ge \binom{t}{2} - (\Lambda + a_i - 2) + (\lceil \alpha\Lambda \rceil - 1)|N_0|.
\end{align*}
Since $u$ is contained in fewer than $\binom{t}{2}$ triangles in $G''$, it follows that:
\[
(\lceil \alpha\Lambda \rceil - 1)|N_0| < \Lambda + a_i - 2,
\]
which yields $|N_0| \le \frac{\Lambda + a_i - 2}{\lceil \alpha\Lambda \rceil - 1}$, that is, $|N_G(u_i^j) \cap (U \setminus U_i)|\le \frac{\Lambda + a_i - 2}{\lceil \alpha\Lambda \rceil - 1}$.
\end{proof}

Continuing the proof of Lemma \ref{b}, we assume for the sake of contradiction that $b > 6\alpha\Lambda + 7$ (respectively, $b > 4\alpha\Lambda + 5$ provided $\Lambda \ge 8$). We first construct $X_0$, then define $J$ and verify the contradiction.

\medskip
\noindent \textbf{The Construction of the independent set.}

We construct an independent set $X_0$ of size $\lceil 2\alpha\Lambda \rceil + 2$ in $U$ as follows.

Let $|\mathcal{U}_1| = b'$, so $|\mathcal{U}_2| = b - b'$. Order $\mathcal{U}_1$ as $(U_{\gamma_1}, U_{\gamma_2}, \ldots, U_{\gamma_{b'}})$ with $a_{\gamma_1} \leq a_{\gamma_2} \leq \cdots \leq a_{\gamma_{b'}}$, where $a_{\gamma_i}=|U_{\gamma_i}|$.  For convenience, we write $V_j = U_{\gamma_j}$ for $1 \leq j \leq b'$.

\begin{enumerate}
    \item For each $U_i \in \mathcal{U}_2$, choose $w_i \in U_i \cap Z$ and initialize $X_0 = \{w_i : U_i \in \mathcal{U}_2\}$.
    \item If $|X_0| \ge \lceil 2\alpha\Lambda \rceil + 2$, retain exactly $ \lceil 2\alpha\Lambda \rceil + 2$ vertices in $X_0$ and terminate. Otherwise, set $k = 1$ and $R_1 = U \setminus \bigcup_{U_i \in \mathcal{U}_2} U_i$.
    \item Let $\ell_k = \min \{ j : V_j \cap R_k \neq \emptyset \}$. Pick $v_k \in V_{\ell_k} \cap R_k$ and set $X_0 \leftarrow X_0 \cup \{v_k\}$.
    \item Set $R_{k+1} = R_k \setminus (V_{\ell_k} \cup N_G(v_k))$.
    \item If $|X_0| = \lceil 2\alpha\Lambda \rceil + 2$ or $R_{k+1} = \emptyset$, terminate. Otherwise, increment $k$ and return to Step 3.
\end{enumerate}

Since the neighbors of any vertex in $Z$ are restricted to its respective part $U_i$, and we select at most one vertex from each such part in Step 1, $X_0$ is an independent set at the end of the first step. The greedy selection in Steps 3--5 maintains this independence by removing the neighborhood of each newly added vertex.

\medskip
\noindent \textbf{The Feasibility of the construction.}

It remains to verify that the algorithm does not terminate prematurely.

\medskip
\noindent\textbf{Case 1: $\Lambda \leq 7$.}

In this case, $b >  6\alpha\Lambda + 7$. We prove by induction that $\ell_k \leq 3k - 2$ and $R_k \neq \emptyset$ for all $k \leq \lceil 2\alpha\Lambda \rceil + 2 - (b - b')$.

Since $b >  6\alpha\Lambda + 7$, we have
\[
3(\lceil 2\alpha\Lambda \rceil + 2 - (b - b')) - 2 \le 6\alpha\Lambda +7 - 3(b - b') < b - (b - b') = b'.
\]

For $k = 1$, clearly $\ell_1 = 1 \le 3 \times 1-2$.
Suppose the statement holds for all $k \leq k_0 - 1$. We now prove it for $k = k_0$.  By Claim~\ref{m},
\begin{align*}
|(V_{\ell_k} \cup N_G(v_k)) \cap U|
&= |V_{\ell_k}| + |N_G(v_k) \cap (U \setminus V_{\ell_k})| \\
&\leq a_{\gamma_{3k-2}} + \frac{\Lambda + a_{\gamma_{3k-2}} - 2}{\lceil \alpha\Lambda \rceil - 1} \\
&= \left(1 + \frac{1}{\lceil \alpha\Lambda \rceil - 1}\right) a_{\gamma_{3k-2}} + \frac{\Lambda - 2}{\lceil \alpha\Lambda \rceil - 1}.
\end{align*}
For $2 \le \Lambda \le 7$, we have $\lceil \alpha\Lambda \rceil \geq 2$. Then
\[
1 + \frac{1}{\lceil \alpha\Lambda \rceil - 1} \leq 1 + 1 = 2
\quad \text{and} \quad
\frac{\Lambda - 2}{\lceil \alpha\Lambda \rceil - 1} < \frac{1}{\alpha} = \frac{7}{4}.
\]
Since $|(V_{\ell_k} \cup N_G(v_k)) \cap U|$  is an integer, we conclude  that $$|(V_{\ell_k} \cup N_G(v_k)) \cap U| \leq 2a_{\gamma_{3k-2}} + 1.$$
As the sequence $\{a_{\gamma_j}\}_{j=1}^{b'}$ is non-decreasing and $a_{\gamma_j} \geq 1$ for all $j$, we have $$2a_{\gamma_{3k-2}} + 1 \leq a_{\gamma_{3k-2}} + a_{\gamma_{3k-1}} + a_{\gamma_{3k}}.$$ Consequently, summing these bounds and using the fact that the sets $V_j$ are pairwise
 disjoint, we obtain:
\[
\left| \bigcup_{k=1}^{k_0-1} (V_{\ell_k} \cup N_G(v_k)) \cap U \right|
\leq \sum_{k=1}^{k_0-1} \left| (V_{\ell_k} \cup N_G(v_k)) \cap U \right|
\leq \sum_{j=1}^{3k_0-3} a_{\gamma_j}
= \left| \bigcup_{j=1}^{3k_0-3} V_j \right|
< \left| \bigcup_{j=1}^{3k_0-2} V_j \right|.
\]

By the definition of $R_{k}$, we obtain
\[
\left| R_{k_0} \cap \left( \bigcup_{j=1}^{3k_0-2} V_j \right) \right|
\geq \left| \bigcup_{j=1}^{3k_0-2} V_j \right| - \left| \bigcup_{k=1}^{k_0-1} (V_{\ell_k} \cup N_G(v_k)) \cap U \right|
> 0,
\]
which implies  $\ell_{k_0} \leq 3k_0 - 2$ and $R_{k_0} \neq \emptyset$.

\medskip
\noindent\textbf{Case 2: $\Lambda \geq 8$.}

In this case, $b > 4\alpha\Lambda + 5$. We prove by induction that $\ell_k \leq 2k - 1$ and $R_k \neq \emptyset$ for all $k \leq \lceil 2\alpha\Lambda \rceil + 2 - (b - b')$.

Since $b > 4\alpha\Lambda + 5$,
\[
2(\lceil 2\alpha\Lambda \rceil + 2 - (b - b')) - 1 \le  4\alpha\Lambda  - 2(b - b')+5 < b - (b - b') = b'.
\]

For $k = 1$, clearly $\ell_1 = 1 \leq  2 \times 1 - 1$.

Assume the statement holds for all $k \leq k_0 - 1$. By Claim~\ref{m},
\begin{align*}
|(V_{\ell_k} \cup N_G(v_k)) \cap U|
&= |V_{\ell_k}| + |N_G(v_k) \cap (U \setminus V_{\ell_k})| \\
&\leq a_{\gamma_{2k-1}} + \frac{\Lambda + a_{\gamma_{2k-1}} - 2}{\lceil \alpha\Lambda \rceil - 1} \\
&= \left(1 + \frac{1}{\lceil \alpha\Lambda \rceil - 1}\right) a_{\gamma_{2k-1}} + \frac{\Lambda - 2}{\lceil \alpha\Lambda \rceil - 1}.
\end{align*}

For $\Lambda \geq 8$, we have $\lceil \alpha\Lambda \rceil \geq 5$. So
\[
1 + \frac{1}{\lceil \alpha\Lambda \rceil - 1} \leq 1 + \frac{1}{4} = \frac{5}{4}
\quad \text{and} \quad
\frac{\Lambda - 2}{\lceil \alpha\Lambda \rceil - 1} < \frac{1}{\alpha} = \frac{7}{4}.
\]

Since the sequence $\{a_{\gamma_j}\}_{j=1}^{b'}$ is non-decreasing with $a_{\gamma_j} \geq 1$ for all $j$, we can bound the size as follows:
\[
|(V_{\ell_k} \cup N_G(v_k)) \cap U|
< \frac{5}{4} a_{\gamma_{2k-1}} + \frac{7}{4}
\leq 2a_{\gamma_{2k-1}} + 1
\leq a_{\gamma_{2k-1}} + a_{\gamma_{2k}} + 1.
\]

As this quantity is an integer, it follows that
\[
|(V_{\ell_k} \cup N_G(v_k)) \cap U| \leq a_{\gamma_{2k-1}} + a_{\gamma_{2k}}.
\]

Consequently, summing these bounds and using the fact that the sets $V_j$ are pairwise disjoint, we deduce:
\[
\left| \bigcup_{k=1}^{k_0-1} (V_{\ell_k} \cup N_G(v_k)) \cap U \right|
\leq \sum_{k=1}^{k_0-1} \left| (V_{\ell_k} \cup N_G(v_k)) \cap U \right|
\leq \sum_{k=1}^{k_0-1} (a_{\gamma_{2k-1}} + a_{\gamma_{2k}})
= \sum_{j=1}^{2k_0-2} a_{\gamma_j}
< \left| \bigcup_{j=1}^{2k_0-1} V_j \right|.
\]

By the definition of $R_{k}$, we obtain
\[
\left| R_{k_0} \cap \left( \bigcup_{j=1}^{2k_0-1} V_j \right) \right|
\geq \left| \bigcup_{j=1}^{2k_0-1} V_j \right| - \left| \bigcup_{k=1}^{k_0-1} (V_{\ell_k} \cup N_G(v_k)) \cap U \right|
> 0,
\]
which  implies $\ell_{k_0} \leq 2k_0 - 1$ and $R_{k_0} \neq \emptyset$, completing the induction.

\medskip
In both cases, we successfully construct an independent set $X_0$ of size exactly $\lceil 2\alpha\Lambda \rceil + 2$ in $U$.

Based on the independent set $X_0$, we will construct a graph $J$ on $V(G)$ such that $|E(J)| < |E(G)|$. To ensure that $J$ remains a valid candidate for Problem \ref{13}, that is, every vertex in $J$ is still contained in at least $\binom{t}{2}$ triangles, we first establish the following crucial claim.

\begin{claim}\label{nx}
    For any vertex $v \in B_1$ and any set  $\{x_1, \dots, x_{\zeta}\} \subseteq X_0$, we have
    \[
    \sum_{j=1}^{\zeta} |N(v) \cap N(x_j) \cap V(\mathcal{D})| \le \zeta(\Lambda - \lceil \alpha\Lambda \rceil) + \Lambda \le (1-\alpha)\zeta\Lambda + \Lambda.
    \]
\end{claim}

\begin{proof}
    For each $1 \le j \le \zeta$, let $x_j \in U_{i_j}$ and $E_j = V(\mathcal{D}_{i_j})$. Since the sets $U_i$ are formed in distinct iterations, we may assume without loss of generality that $i_1 < i_2 < \dots < i_{\zeta}$. By the definition of our iterative procedure, all neighbors of $x_j$ in $E_j$ are removed in iteration $i_j$; that is, $N(x_j) \cap E_j \subseteq E_j \setminus E_{j+1}$. Moreover, by \eqref{equaas}, $x_j$ satisfies $f_{i_j}(x_j) \ge \lceil \alpha\Lambda \rceil$. Then we have
    \[
    |N(x_j) \cap (V(\mathcal{D}) \setminus E_j)| = f_1(x_j) - f_{i_j}(x_j) \le \Lambda - \lceil \alpha\Lambda \rceil.
    \]
So we have
  %  We now partition the common neighborhood of $v$ and $x_j$ within $V(\mathcal{D})$ into two parts:
    \begin{align*}
        |N(v) \cap N(x_j) \cap V(\mathcal{D})|
        &= |N(v) \cap N(x_j) \cap E_j| + |N(v) \cap N(x_j) \cap (V(\mathcal{D}) \setminus E_j)| \\
        &\le |N(v) \cap (E_j \setminus E_{j+1})| + |N(x_j) \cap (V(\mathcal{D}) \setminus E_j)| \\
        &\le |N(v) \cap (E_j \setminus E_{j+1})| + (\Lambda - \lceil \alpha\Lambda \rceil).
    \end{align*}
    Summing this over $j = 1, \dots, \zeta$, and noting that the sets $E_j \setminus E_{j+1}$ are pairwise disjoint subsets of $V(\mathcal{D})$, we obtain
    \begin{align*}
        \sum_{j=1}^{\zeta} |N(v) \cap N(x_j) \cap V(\mathcal{D})|
        &\le \sum_{j=1}^{\zeta} |N(v) \cap (E_j \setminus E_{j+1})| + \zeta(\Lambda - \lceil \alpha\Lambda \rceil) \\
        &\le |N(v) \cap V(\mathcal{D})| + \zeta(\Lambda - \lceil \alpha\Lambda \rceil) \\
        &\le \Lambda + \zeta(\Lambda - \lceil \alpha\Lambda \rceil),
    \end{align*}
    where the last step follows from the fact that $f_1(v) \le \Lambda$.
    The second inequality in the lemma statement follows immediately from $\Lambda - \lceil \alpha\Lambda \rceil \le (1-\alpha)\Lambda$.
\end{proof}

\subsection*{The Construction of the Graph \texorpdfstring{$J$}{J}}

We construct a graph $J$ on the vertex set $V(G)$ to demonstrate that $G$ cannot be extremal. Let $x = |X_0| - 1$ if $\Lambda = 2$, and $x = |X_0|$ otherwise. Let $X \subseteq X_0$ be a subset of cardinality $x$, and write $X = \{y_1, \ldots, y_x\}$. Furthermore, let $D_1 = \{y_{x+1}, \ldots, y_{x+t+1}\}$ and define $W = X \cup D_1$, endowed with the ordering $y_1 < y_2 < \cdots < y_{x+t+1}$.

The graph $J$ is obtained from $G$ by applying the following three modifications to its edge set:
\begin{enumerate}[label=(\roman*)]
    \item \textbf{Reconfiguration of $X$ and $V(\mathcal{D})$:} The set $X$, which is independent in $G$, is transformed into a clique $K_x$. Simultaneously, all edges in $G$ between $X$ and the union of cliques $V(\mathcal{D}) = \bigcup_{i=1}^d D_i$ are removed.
    \item \textbf{Isolation of Redundant Cliques:} For each $j \in \{2, \dots, d\}$, the clique $D_j$ is made isolated by removing all edges between $D_j$ and  $V(G) \setminus D_j$.
    \item \textbf{Neighborhood Shifting:} For any vertex $v \in B_1 \setminus X$, we standardize its adjacencies into $W = X \cup D_1$. Specifically, $N_J(v) \cap W$ is defined to consist of $N_G(v) \cap X$ together with the $f_1(v)$ vertices of $W \setminus (N_G(v) \cap X)$ that are smallest under the ordering of $W$.
\end{enumerate}

This construction replaces the dispersed adjacencies of $ B_1 \setminus  X$ with a standardized distribution concentrated on $X \cup D_1$, while simultaneously altering the internal structure of $X$. To confirm that this modification reduces the total number of edges, we now evaluate  $|E(J)| - |E(G)|$.

Observe that for each $v \in B_1 \setminus X$, the neighborhood shifting in step (iii) is degree-preserving relative to $V(\mathcal{D}) \cup X$. Specifically, the construction ensures $|N_J(v) \cap W| = |N_G(v) \cap X| + f_1(v) = |N_G(v) \cap (X \cup V(\mathcal{D}))|.$
Consequently, the contribution of these adjacencies to the net change $|E(J)| - |E(G)|$ vanishes. The difference is therefore determined entirely by the internal edges of $X$ and the edges incident to $X$ in $V(\mathcal{D})$:
\begin{align*}
    |E(J)| - |E(G)| &\leq |E(J[X])| - \sum_{j=1}^x |N_G(y_j) \cap V(\mathcal{D})| \\
    &= \binom{x}{2} - \sum_{j=1}^x f_1(y_j) \\
    &= \sum_{j=1}^x \left( \frac{x-1}{2} - f_1(y_j) \right).
\end{align*}

To establish $|E(J)| < |E(G)|$, it suffices to show that $f_1(y_j) > (x-1)/2$ for all $1 \le j \le x$. By the construction of $X$, each $y_j \in X$ resides in some set $U_i$ with $i \le b_1$, which makes the threshold condition in \eqref{def:bi} applicable to $f_i(y_j)$. Together with the monotonicity property $f_1(y_j) \ge f_i(y_j)$ from \eqref{equaas}, we obtain the following two cases.
\begin{itemize}
    \item If $\Lambda \ge 3$, the lower bound on $f_i(y_j)$ from \eqref{def:bi} ensures that
    \[ f_1(y_j) \ge f_i(y_j) > \frac{\lceil 2\alpha\Lambda \rceil + 1}{2}. \]
    Since $x \le \lceil 2\alpha\Lambda \rceil + 2$, it follows immediately that $f_1(y_j) > (x-1)/2$.
    \item If $\Lambda = 2$, we have $x = 4$ by construction. Here, \eqref{def:bi} and \eqref{equaas} imply $f_1(y_j) \ge f_i(y_j) = 2$, which strictly exceeds $(x-1)/2 = 3/2$.
\end{itemize}
In either case, the inequality $|E(J)| < |E(G)|$ holds.

    It remains to verify that every vertex $v \in V(J)$ belongs to at least $\binom{t}{2}$ triangles in $J$. For $v \in V(G)$, let $t_{ij}(v)$ and $t'_{ij}(v)$ denote the number of triangles containing $v$ in $G$ and $J$, respectively, with exactly $i$ vertices in $V(\mathcal{D})$ and $j$ vertices in $X \setminus \{v\}$, where $0 \le i, j, i+j \le 2$. By the construction, any triangle in $G$ containing $v$ with no vertices in $V(\mathcal{D}) \cup X$ is preserved; thus, $t'_{00}(v) = t_{00}(v)$ for all $v \in V(G)$. We proceed by considering the location of $v$.

\noindent \textbf{Case 1:} $v \in V(\mathcal{D})$. \\
By the construction of $J$, each $D_i$ is preserved as a $(t+1)$-clique. Thus, any $v \in D_i$ is contained in $\binom{t}{2}$ triangles within $D_i$, satisfying the requirement.

\medskip
\noindent \textbf{Case 2:} $v \in V(G) \setminus (V(\mathcal{D}) \cup B_1)$. \\
For such $v$, our construction ensures $N_J(v) = N_G(v)$. Since no edges within $V(G) \setminus V(\mathcal{D})$ are removed, we have $G[N_G(v)] \subseteq J[N_J(v)]$. It follows that the triangle degree of $v$ in $J$ is at least its  triangle degree in $G$, which is at least $\binom{t}{2}$.

\medskip
\noindent \textbf{Case 3:} $v \in X$. \\
Recall that $X$ is an independent set in $G$, so $t_{ij}(v) = 0$ for any $j \ge 1$. Consequently, the  triangle degree of $v$ in $G$ consists only of terms $t_{20}, t_{10},$ and $t_{00}$. In $J$, the set $X$ becomes a clique, and all edges in $G$ between $X$ and  $V(\mathcal{D}) = \bigcup_{i=1}^d D_i$ are removed. The total number of triangles containing $v$ in $J$ is given by:
\begin{equation} \label{eq:case3_main}
    t'_{02}(v) + t'_{01}(v) + t'_{00}(v) \ge \binom{x-1}{2} + \sum_{u \in N_J(v) \setminus X} (|N_J(u) \cap X| - 1) + t'_{00}(v).
\end{equation}

We show that each term in \eqref{eq:case3_main} is bounded from below by its counterpart in $G$. First, the condition $x \ge \lceil 2\alpha\Lambda \rceil + 1$ immediately implies
\[
\binom{x-1}{2} \ge \binom{\lceil 2\alpha\Lambda \rceil}{2} \ge \binom{\Lambda}{2} \ge t_{20}(v).
\]
For the second term, recall that the neighborhood shifting rule ensures $|N_J(u) \cap X| \ge \min\{ |N_G(u) \cap (V(\mathcal{D}) \cup X)|, x \}$ for any $u \in V(G) \setminus X$. Since $v \in N_G(u) \cap X$ and $x-1 \ge \lceil 2\alpha\Lambda \rceil > \Lambda \ge |N_G(u) \cap V(\mathcal{D})|$, it follows that $|N_J(u) \cap X| - 1 \ge |N_G(u) \cap V(\mathcal{D})|$. Combined with the containment $N_G(v) \setminus V(\mathcal{D}) \subseteq N_J(v) \setminus X$, we obtain
\begin{align*}
    \sum_{u \in N_J(v) \setminus X} (|N_J(u) \cap X| - 1) &\ge \sum_{u \in N_G(v) \setminus V(\mathcal{D})} |N_G(u) \cap V(\mathcal{D})| \ge t_{10}(v).
\end{align*}
Finally, noting that $t'_{00}(v) = t_{00}(v)$, we conclude that
\[
t'_{02}(v) + t'_{01}(v) + t'_{00}(v) \ge t_{20}(v) + t_{10}(v) + t_{00}(v) \ge \binom{t}{2},
\]
as desired.

\medskip
\noindent \textbf{Case 4:} $v \in B_1 \setminus X$. \\
For $v \in B_1 \setminus X$, we categorize the triangles containing $v$ in $J$ based on the partition $\{V(\mathcal{D}), X, V(G) \setminus (V(\mathcal{D})\cup X)\}$. Since $D_2, \dots, D_d$ are isolated in $J$ and there are no adjacencies between $D_1$ and $X$, any triangle incident to $v$ cannot simultaneously contain vertices from both $V(\mathcal{D})$ and $X$. Consequently, the cross-terms $t'_{ij}(v)$ with $i \ge 1, j \ge 1$ (specifically $t'_{11}$) vanish. The  triangle-degree of $v$ in $J$ decomposes into  the following four components:
\begin{align*}
    \sum_{0 \le i, j \le 2} t'_{ij}(v)
    &= \underbrace{\binom{|N_J(v) \cap D_1|}{2}}_{t'_{20}(v)} + \underbrace{\binom{|N_J(v) \cap X|}{2}}_{t'_{02}(v)} \\
    &\quad + \underbrace{\sum_{u \in N_J(v) \setminus (V(\mathcal{D}) \cup X)} |N_J(u) \cap N_J(v) \cap (D_1 \cup X)|}_{t'_{10}(v) + t'_{01}(v)} + t'_{00}(v).
\end{align*}
The right-hand side accounts for all possible triangle configurations:
\begin{itemize}
    \item The first two binomial terms count triangles formed entirely within $\{v\} \cup D_1$ and $\{v\} \cup X$, representing $t'_{20}$ and $t'_{02}$, respectively.
    \item The summation term aggregates triangles containing $v$, one ``external" vertex $u \in V(G) \setminus (V(\mathcal{D}) \cup X)$, and a third vertex $w \in D_1 \cup X$. Because $D_1$ and $X$ are disjoint, the cardinality $|N_J(u) \cap N_J(v) \cap (D_1 \cup X)|$ is simply the sum $|N_J(u) \cap N_J(v) \cap D_1| + |N_J(u) \cap N_J(v) \cap X|$. These two parts correspond exactly to the counts $t'_{10}(v)$ and $t'_{01}(v)$.
    \item $t'_{00}(v)$ corresponds to triangles whose vertices lie entirely outside $V(\mathcal{D}) \cup X$, which remain unchanged from $G$.
\end{itemize}
Note that $N_J(v) \setminus (V(\mathcal{D}) \cup X) = N_G(v) \setminus (V(\mathcal{D}) \cup X)$, as the construction of $J$ preserves all adjacencies within the complement of $V(\mathcal{D}) \cup X$. Furthermore, the neighborhood shifting ensures that for each $u \in V(G) \setminus (V(\mathcal{D}) \cup X)$, the common neighborhood of $u$ and $v$ within the target set $D_1 \cup X$ is no smaller than their original common neighborhood in $V(\mathcal{D}) \cup X$. Specifically, $|N_J(u) \cap N_J(v) \cap (D_1 \cup X)| \ge |N_G(u) \cap N_G(v) \cap (V(\mathcal{D}) \cup X)|$. Summing over $u \in N_G(v) \setminus (V(\mathcal{D}) \cup X)$ and noting that $t'_{00}(v) = t_{00}(v)$, we obtain the following inequality chain:
\begin{align*}
    &\sum_{u \in N_J(v) \setminus (V(\mathcal{D}) \cup X)} |N_J(u) \cap N_J(v) \cap (D_1 \cup X)| + t'_{00}(v)\\
    &\ge \sum_{u \in N_G(v) \setminus (V(\mathcal{D}) \cup X)} |N_G(u) \cap N_G(v) \cap (V(\mathcal{D}) \cup X)| + t_{00}(v) \\
    &= t_{10}(v) + t_{01}(v) + t_{00}(v).
\end{align*}
Since $X$ is an independent set in $G$, we have $t_{02}(v) = 0$. Consequently, to establish that $v$ belongs to at least $\binom{t}{2}$ triangles in $J$, it suffices to verify that the triangles formed entirely within $D_1 \cup X$ satisfy:
\begin{equation} \label{eq:case4_final_goal}
\binom{|N_J(v) \cap D_1|}{2} + \binom{|N_J(v) \cap X|}{2} \ge t_{20}(v) + t_{11}(v).
\end{equation}

We distinguish two cases based on the neighborhood of $v$ in $D_1$ under the graph $J$.

\medskip
\noindent \textbf{Case 4.1:} $N_J(v) \cap D_1 = \emptyset$. \\
By the neighborhood shifting rule, if $v$ has no neighbors in $D_1$, then all its adjacencies in $V(\mathcal{D}) \cup X$ must have been concentrated into $X$. Thus, we have $|N_J(v) \cap X| = |N_G(v) \cap (X \cup V(\mathcal{D}))|$. Recalling that $t_{02}(v) = 0$ as $X$ is an independent set in $G$, we obtain:
\begin{align*}
    \binom{|N_J(v) \cap D_1|}{2} + \binom{|N_J(v) \cap X|}{2} &= \binom{|N_G(v) \cap (X \cup V(\mathcal{D}))|}{2} \\
    &= \binom{|N_G(v) \cap X|}{2} + |N_G(v) \cap X| \cdot |N_G(v) \cap V(\mathcal{D})| + \binom{|N_G(v) \cap V(\mathcal{D})|}{2} \\
    &\ge t_{11}(v) + t_{20}(v).
\end{align*}
The last inequality follows because $|N_G(v) \cap X| \cdot |N_G(v) \cap V(\mathcal{D})|$ provides a trivial upper bound on the number of pairs $\{u, d\}$ with $u \in N_G(v) \cap X$ and $d \in N_G(v) \cap V(\mathcal{D})$ that form a triangle with $v$.

\medskip
\noindent \textbf{Case 4.2:} $N_J(v) \cap D_1 \neq \emptyset$. \\
In this case, the shifting rule implies that the neighborhood of $v$ in $J$ fully contains $X$.
Recall that $x = |X|$.
We have $|N_J(v) \cap X| = x$ and $|N_J(v) \cap D_1| = |N_G(v) \cap (X \cup V(\mathcal{D}))| - x$. Using the identity $\binom{A}{2} + \binom{B}{2} = \binom{A+B}{2} - AB$, the left-hand side of \eqref{eq:case4_final_goal} becomes:
\begin{align*}
    \binom{|N_J(v) \cap D_1|}{2} + \binom{|N_J(v) \cap X|}{2} &= \binom{|N_G(v) \cap (X \cup V(\mathcal{D}))| - x}{2} + \binom{x}{2} \\
    &= \binom{|N_G(v) \cap (X \cup V(\mathcal{D}))|}{2} - x\left( |N_G(v) \cap (X \cup V(\mathcal{D}))| - x \right).
\end{align*}
On the other hand, by the definition of $t_{11}(v)$ and applying Claim \ref{nx} to bound the sum of common neighborhoods, we have
\begin{align*}
    t_{20}(v) + t_{11}(v) &\leq \binom{|N_G(v) \cap V(\mathcal{D})|}{2} + \sum_{u \in N_G(v) \cap X} |N_G(v) \cap N_G(u) \cap V(\mathcal{D})| \\
    &\le \binom{|N_G(v) \cap V(\mathcal{D})|}{2} + |N_G(v) \cap X|(\Lambda - \lceil \alpha\Lambda \rceil) + \Lambda.
\end{align*}

In the case $\Lambda = 2$, we have $x = \lceil 2\alpha\Lambda \rceil + 1 = 4$. Since $N_J(v) \cap D_1 \neq \emptyset$, the neighborhood shifting rule ensures that $|N_J(v) \cap X | = |X| = 4$. Thus, the left-hand side of the target inequality satisfies
\begin{align*}
    \binom{|N_J(v) \cap D_1|}{2} + \binom{|N_J(v) \cap X|}{2} \ge \binom{4}{2}= 6.
\end{align*}
On the other hand, for $\Lambda = 2$, we have $\Lambda - \lceil \alpha\Lambda \rceil = 2 - 2 = 0$. Consequently, the estimate for $G$ from Claim \ref{nx} simplifies to:
\begin{align*}
    t_{20}(v) + t_{11}(v) &\le \binom{|N_G(v) \cap V(\mathcal{D})|}{2} + (\Lambda - \lceil \alpha\Lambda \rceil) |N_G(v) \cap X| + \Lambda \\
    &\le \binom{2}{2} + 0 + 2 = 3.
\end{align*}
Comparing these bounds, we obtain $\binom{|N_J(v) \cap D_1|}{2} + \binom{|N_J(v) \cap X|}{2} \ge 6 > 3 \ge t_{20}(v) + t_{11}(v)$, confirming the
minimum triangle-degree condition.

In the case  $\Lambda \ge 3$, we have $x = \lceil 2\alpha\Lambda \rceil + 2$. Let $y = |N_G(v) \cap (X \cup V(\mathcal{D}))|$ and note that $y$ is constrained to the interval $\lceil 2\alpha\Lambda \rceil + 2 \le y \le \lceil 2\alpha\Lambda \rceil + \Lambda + 2$. Based on the construction of $J$, the number of triangles in $J$ formed within $D_1 \cup X$ is
\begin{align*}
    \binom{|N_J(v) \cap D_1|}{2} + \binom{|N_J(v) \cap X|}{2} &= \binom{y - x}{2} + \binom{x}{2} \\
    &= \binom{y}{2} - x(y - x).
\end{align*}
On the other hand, applying Claim \ref{nx}, we estimate the triangle degree of $v$ in $G$:
\begin{align*}
    t_{20}(v) + t_{11}(v) &\le \binom{|N_G(v) \cap V(\mathcal{D})|}{2} + (1 - \alpha)\Lambda |N_G(v) \cap X| + \Lambda \\
    &= \binom{|N_G(v) \cap V(\mathcal{D})|}{2} + (1 - \alpha)\Lambda (y - |N_G(v) \cap V(\mathcal{D})|) + \Lambda.
\end{align*}
Let $k = |N_G(v) \cap V(\mathcal{D})|$ and set $\Phi (k)=\binom{k}{2} + (1 - \alpha)\Lambda (y-k) + \Lambda$. Since $\Phi (k)$ is a convex function of $k$ over the interval $[y-x, \Lambda]$, its maximum is attained at the boundaries. Thus, to establish \eqref{eq:case4_final_goal}, it suffices to verify the inequality
\[ \binom{y}{2} - x(y - x) \ge \max \{ \Phi(y-x), \Phi(\Lambda) \} \]
for the following two boundary cases:

\smallskip
\noindent \textit{Boundary Case 1: $k = y - x$.} \\
A direct expansion of the difference yields:
\begin{align*}
    \left[ \binom{y}{2} - x(y - x) \right] - \left[ \binom{y - x}{2} + (1 - \alpha)\Lambda x + \Lambda \right] &= \binom{x}{2} - (1 - \alpha)\Lambda x - \Lambda \\
    &\ge \binom{2\alpha\Lambda + 2}{2} - (1 - \alpha)\Lambda (2\alpha\Lambda + 2) - \Lambda \\
    &= (4\alpha^2 - 2\alpha)\Lambda^2 + (5\alpha - 3)\Lambda + 1,
\end{align*}
which is non-negative for $\alpha = 4/7$ and $\Lambda \ge 3$.

\smallskip
\noindent \textit{Boundary Case 2: $k = \Lambda$.} \\
We define the difference as a quadratic function $\Psi(y)$:
\begin{align*}
    \Psi(y) &= \binom{y}{2} - x(y - x) - \left[ \binom{\Lambda}{2} + (1 - \alpha)\Lambda (y - \Lambda) + \Lambda \right] \\
    &\ge \frac{1}{2}y^2 - \left( (1 + \alpha)\Lambda + \frac{5}{2} \right)y + \left( 4\alpha^2 - \alpha + \frac{1}{2} \right)\Lambda^2 + \left( 8\alpha - \frac{1}{2} \right)\Lambda + 4.
\end{align*}
The minimum of this parabola is attained at $y^* = (1 + \alpha)\Lambda + \frac{5}{2}$. Completing the square shows:
\[ \Psi(y) \ge \Psi(y^*) = \left( \frac{7}{2}\alpha^2 - 2\alpha \right)\Lambda^2 + \left( \frac{11}{2}\alpha - 3 \right)\Lambda + \frac{7}{8}, \]
which is strictly positive for $\alpha = 4/7$ and $\Lambda \ge 3$.

\medskip
\noindent In all cases, every vertex $v \in V(J)$ is contained in at least $\binom{t}{2}$ triangles. Since $|E(J)| < |E(G)|$, this contradicts the extremality of $G$, thereby completing the proof of Lemma \ref{b}.

\end{proof}

\subsection{Lemma \ref{d1}}
\begingroup
\it
For any $1\le i\le \xi $, if $f_i(u_i) = 1$, then $b_i \le \frac{3}{4}t$.
\endgroup

\begin{proof}
As we did in Lemma \ref{b}, it suffices to establish the result for $i = 1$.
Suppose, for the sake of contradiction, that $\Lambda=f_1(u_1) = 1$ but $b_1 > \frac{3}{4}t$. In this case, each iteration of our procedure removes exactly one clique from the collection. Let $\mathcal{D} = \{D_1, \dots, D_d\}$, where $d = b_1 > \frac{3}{4}t$.

For each $1 \le i \le d$, let $D_i$ be the unique clique in  $\mathcal{D}_i \setminus \mathcal{D}_{i+1}$. Consider a sequence of graphs $L_0, L_1, \dots, L_d$, where $L_0 = G$ and, for each $i \ge 1$,
\[
L_i = \mathcal{G}(L_{i-1}, D_i, D_i).
\]
Since $\Lambda = 1$, each vertex $u \notin D_i$ is adjacent to at most one vertex in $D_i$. Thus, the transformation $L_i$ effectively redirects all edges between $V(G) \setminus D_i$ and $D_i$ to the first vertex of $D_i$. Let $L = L_d$ and by Lemma \ref{G}, every vertex in $L$ is contained in at least $\binom{t}{2}$ triangles.

For each $1 \le i \le d$, let $W_i = \{v \in V(G) \setminus D_i : N_L(v) \cap D_i \neq \emptyset\}$ and $e_i=|E(L[W_i])|$.
Since $D_i \in \mathcal{A}_1$, its total excess degree satisfies $\sum_{v \in D_i} (d_L(v) - t) \le \theta$, where $\theta = \frac{\sqrt{33}}{2\sqrt{928}}(t+1)$. Under the assumption $\Lambda=1$, each vertex $u \in W_i$ has exactly one neighbor in $D_i$. Consequently, $|W_i| = \sum_{v \in D_i} (d_L(v) - t) \le \theta$.

Furthermore, the condition $\Lambda=1$ implies that for each $1 \le i \le d$, each $D_i$ contains a unique vertex $z_i$ such that $d_L(z_i) \ge t+1$, while all other vertices in $D_i$ have degree $t$ in $L$. Thus, $N_L(u) \cap D_i = \{z_i\}$ for all $u \in W_i$. By our construction, the sets $W_1, \dots, W_d$ are pairwise disjoint. Since $L$ inherits the extremality of $G$, every edge of $L$ must be contained in a triangle. Any triangle containing an edge between $z_i$ and a vertex in $W_i$ must have its third vertex in $W_i$, implying that $|N_L(u) \cap W_i| \ge 1$ for every $u \in W_i$ and then $W_i$ is not an independent set. Consequently $|W_i| \ge 2$ for each $1 \le i \le d$.

\begin{claim} \label{w1}

There exist distinct indices $i, j \in \{1, \dots, d\}$ and a vertex $w \in W_i$ such that $|N_L(w) \cap W_j| \le 1$ and $e_i \le e_j$.
\end{claim}

\begin{proof}
Suppose, for the sake of contradiction, that for all $i, j\in \{1, \dots, d\}$ with $e_i \le e_j$, every vertex $w \in W_i$ satisfies $|N_L(w) \cap W_j| \ge 2$. Without loss of generality, assume the cliques are ordered such that $e_1 \le e_2 \le \dots \le e_d$. Under this assumption, for any $1 \le i < j \le d$ and any $w \in W_i$, we have $|N_L(w) \cap W_j| \ge 2$, which implies $E_L(W_i, W_j) = \sum_{w \in W_i} |N_L(w) \cap W_j| \ge 2|W_i|$.

Recall that for each $u \in W_i$ and any $1\le i\le d$, $|N_L(u) \cap D_i| = 1$ and $|N_L(u) \cap W_i| \ge 1$. Summing over the first $M := d - \lceil t/2 \rceil + 1$ sets $W_i$, we obtain
\begin{align*}
    \sum_{v \in V(G)} (d_L(v) - t) &\ge \sum_{i=1}^{M} \sum_{u \in W_i} (d_L(u) - t) \\
    &\ge \sum_{i=1}^{M} \sum_{u \in W_i} \left( |N_L(u) \cap D_i| + |N_L(u) \cap W_i| + \sum_{j=1}^{i-1} |N_L(u) \cap W_j| +\sum_{j= i+1}^d |N_L(u) \cap W_j|- t \right) \\
    &\ge \sum_{i=1}^{M} \left( |W_i| (2 + 2(d-i) - t) + \sum_{j=1}^{i-1} E_L(W_i, W_j) \right).
\end{align*}
Note that for all $i \le M$, the choice of $M = d - \lceil t/2 \rceil + 1$ ensures that $2(d-i) + 2 - t \ge 0$. Substituting the bounds $|W_i| \ge 2$ and $E_L(W_i, W_j) \ge 2|W_j| \ge 4$ into the summation, the $i$-th term is at least
\[
2(2d - 2i + 2 - t) + 4(i-1) = 4d - 2t.
\]
Summing over $i$ from $1$ to $M$ yields
\[
\sum_{v \in V(G)} (d_L(v) - t) \ge M(4d - 2t) = (d - \lceil t/2 \rceil + 1)(4d - 2t).
\]
Given $d > \frac{3}{4}t$, a straightforward calculation shows that this expression strictly exceeds $\frac{1}{4}(t+1)^2$, a contradiction with Lemma \ref{d}.
\end{proof}

Without loss of generality, assume the condition in Claim \ref{w1} holds for $w_1 \in W_1$ and $W_d$ with $e_1 \le e_d$.  Let $W_1 = \{w_1, \dots, w_a\}$.
%Since $D_1 \in \mathcal{A}_1$, the definition of $\mathcal{A}_1$ combined with the lower bound on $d$ implies
%a < %\frac{\sqrt{33}}{2\sqrt{928}}(t+1) < \frac{3}{4}t \le d.$
We obtain $L'$ from $L$ by executing the following edge modifications:

\begin{itemize}
    \item \textbf{Deletions:}
    \begin{enumerate}
        \item Remove the edge $w_1 z_1$.
        \item For each $j \in \{2, \dots, d\}$, remove all edges between the clique $D_j$ and $V(G) \setminus D_j$ (particularly the edges $\{u z_j : u \in W_j\}$).
    \end{enumerate}

    \item \textbf{Additions:}
    \begin{enumerate}
        \item Add all edges between $w_1$ and $W_d$ that are not already present in $L$.
        \item For each $i \in \{2, \dots, a\}$, if $w_i$ has no neighbors in $W_i$ in the graph $L$, add all edges $\{u w_i : u \in W_i\}$. Otherwise, add all edges $\{u z_1 : u \in W_i\}$.
        \item For each $i \in \{a+1, \dots, d-1\}$, add all edges $\{u z_1 : u \in W_i\}$.
    \end{enumerate}
\end{itemize}

By construction, all adjacencies not explicitly modified above remain unchanged.
We now evaluate the total change in number of the edges by considering three contributions.
First, for each $i \in \{2, \dots, d-1\}$, the removal of edges between $z_i$ and $W_i$ is precisely balanced by the addition of edges between $W_i$ and either $w_i$ or $z_1$, yielding no net change.
Second, consider the adjacencies of $W_d$. In $L'$, we isolate the clique $D_d$ by removing all edges $\{u z_d : u \in W_d\}$, but we simultaneously add all edges between $w_1$ and $W_d$ that were not already present in $L$. The number of edges added is $|W_d| - |N_L(w_1) \cap W_d|$, while $|W_d|$ edges are removed.
Finally, we account for the edge $w_1 z_1$, which is removed and not replaced. Summing these contributions, the total change in the number of edges is
\begin{align*}
|E(L')| - |E(L)| &= \underbrace{(|W_d| - |N_L(w_1) \cap W_d|) - |W_d|}_{\text{from } W_d \text{ and } w_1} - \underbrace{1}_{\text{from } w_1 z_1} \\
&= -1 - |N_L(w_1) \cap W_d|.
\end{align*}
Since $|N_L(w_1) \cap W_d| \ge 0$, we have $|E(L')| \le |E(L)| - 1$. To show that this reduction in the number of edges contradicts the extremality of $L$, it suffices to verify that $L'$ remains a valid candidate for the extremal problem. The following claim ensures that $L'$ still satisfies the required structural constraints.

\begin{claim}\label{l'}
    Every vertex $v\in V(G)$ is contained in at least $\binom{t}{2}$ triangles in $L'$.
\end{claim}

\begin{proof}
Let $W = \bigcup_{i=1}^d W_i$. By the construction, every vertex $u \in V(\mathcal{D})$ is contained in a $(t+1)$-clique in $L'$, thereby satisfying  the minimum triangle-degree condition. It remains to verify the bound for $u \in V(G) \setminus V(\mathcal{D})$.

For $u \in V(G) \setminus V(\mathcal{D})$, let $s_{ij}(u)$ and $s'_{ij}(u)$ denote the number of triangles containing $u$ in $L$ and $L'$, respectively, having exactly $i$ vertices in $V(\mathcal{D})$ and $j$ vertices in $W \setminus \{u\}$. The condition $\Lambda=1$ implies that no vertex has two neighbors in $V(\mathcal{D})$ which implies $s_{20}(u) = 0$.   Moreover, as $W = N_L(V(\mathcal{D})) \setminus V(\mathcal{D})$, any neighbor of $V(\mathcal{D})$ outside $V(\mathcal{D})$ must lie in $W$, which implies $s_{10}(u) = 0$.

\noindent \textbf{General Monotonicity Observation.}
The transformation $L \to L'$ only removes edges incident to $V(\mathcal{D})$ (specifically $w_1 z_1$ and edges $\{z_j v : v \in W_j\}$ for $j \ge 2$), while all adjacencies within $V(G) \setminus V(\mathcal{D})$ are preserved or added. Consequently, $L[V(G) \setminus V(\mathcal{D})]$ is a spanning subgraph of $L'[V(G) \setminus V(\mathcal{D})]$, which immediately yields:
\begin{equation*} \label{eq:base-mono}
s'_{02}(u) \ge s_{02}(u), \quad s'_{01}(u) \ge s_{01}(u), \quad \text{and} \quad s'_{00}(u) \ge s_{00}(u).
\end{equation*}
In view of these relations, the extremality of $L$ ensures the required bound for $u$ in $L'$ provided that
\begin{equation} \label{eq:triangle-sum}
s'_{11}(u) + s'_{02}(u) \ge s_{11}(u) + s_{02}(u).
\end{equation}
In the following analysis, we apply  $s'_{02}(u) \ge s_{02}(u)$ as a default. We shall provide a refined accounting of the difference $s'_{02}(u) - s_{02}(u)$ only in those instances where a potential loss in $s_{11}(u)$ must be specifically compensated.

We verify \eqref{eq:triangle-sum} by considering the following four cases.

\medskip
\noindent \textbf{Case 1:} $u \in V(G) \setminus (V(\mathcal{D}) \cup W)$. \\
By the definition of $W$, the neighborhood of $u$ in $L$ satisfies $N_L(u) \subseteq V(G) \setminus V(\mathcal{D})$. Consequently, every triangle containing $u$ in $L$ lies entirely within $V(G) \setminus V(\mathcal{D})$. Since $E_L(V(G) \setminus V(\mathcal{D})) \subseteq E_{L'}(V(G) \setminus V(\mathcal{D}))$ by the construction, all such triangles are preserved in $L'$, and the required bound for $u$ follows immediately.

\medskip
\noindent \textbf{Case 2:} $u \in W_i$ for some $2 \le i \le d$. \\
In $L$, the vertex $u$ has a unique neighbor $z_i \in V(\mathcal{D})$. By the construction of $L'$, this adjacency is redirected to either $z_1$ or $w_i$ ($w_1$ if $i = d$).

First, suppose $u z_1 \in E(L')$. In this case, $z_1$ inherits the structural role of $z_i$, and the triangles containing $\{u, z_1\}$ are newly formed in $L'$ as a result of our construction. Since $N_{L'}(u) \setminus V(\mathcal{D}) \supseteq N_L(u) \setminus V(\mathcal{D})$ and $N_{L'}(z_1) \cap W \supseteq N_L(z_i) \cap W$, any triangle in $L$ containing $\{u, z_i\}$ corresponds to a triangle in $L'$ containing $\{u, z_1\}$. Consequently, we have
\begin{align*}
    s'_{11}(u) + s'_{02}(u)
    &= |N_{L'}(u) \cap N_{L'}(z_1) \cap W| + s'_{02}(u)  \\
    &\ge |N_{L}(u) \cap N_{L}(z_i) \cap W| + s_{02}(u) \\
    &= s_{11}(u) + s_{02}(u).
\end{align*}

Assume $u z_1 \notin E(L')$. Then $u$ is adjacent to $w_i$ (or $w_1$ if $i = d$). In this scenario, $u$ has no neighbor in $V(\mathcal{D})$ (that is, $s'_{11}(u) = 0$), and the structural role of  $z_i$ is effectively inherited  by a representative in $W_1$. Specifically, for $i=d$, any triangle $\{u, z_d, v\}$ in $L$ must satisfy $v \in W_d$ since $\Lambda=1$. The condition $|N_L(w_1) \cap W_d| \le 1$ ensures that the triangles $\{u, w_1, v\}$ in $L'$ are newly formed and mirror the original $\{u, z_d, v\}$. We thus obtain
\begin{align*}
    s'_{11}(u) + s'_{02}(u)
    = s'_{02}(u) &\ge  |N_{L'}(u) \cap N_{L'}(w^*) \cap W| + s_{02}(u)  \\
    &\ge |N_{L}(u) \cap N_{L}(z_i) \cap W| + s_{02}(u) \\
    &= s_{11}(u) + s_{02}(u),
\end{align*}
where $w^*$ is $w_i$ if $i < d$ and $w_1$ if $i = d$. In all instances, the inequality \eqref{eq:triangle-sum} is satisfied.

\medskip
\noindent \textbf{Case 3:} $u = w_i \in W_1 \setminus \{w_1\}$ for some $2 \le i \le a$. \\
In this case, the only loss of triangles in $L'$ arises from the removal of the edge $w_1 z_1$, which destroys the triangle $\{u, w_1, z_1\}$ (a triangle of type $s_{11}$). Thus, $s'_{11}(u) \ge s_{11}(u) - 1$.  We show that the loss is compensated by new adjacencies.

First, suppose $N_L(u) \cap W_i = \emptyset$. By our construction, all vertices of $W_i$ are made adjacent to $u$ in $L'$, and thus $W_i \subseteq N_{L'}(u)$. Since $W_i$ is not an independent set in $L$, there exists at least one edge in $L[W_i]$, which now forms a new triangle containing $u$ of type $s'_{02}$. Consequently, $s'_{02}(u) \ge s_{02}(u) + 1$, which yields
\begin{align*}
    s'_{11}(u) + s'_{02}(u)
    \ge s_{11}(u) + s_{02}(u).
\end{align*}

Assume $N_L(u) \cap W_i \neq \emptyset$. Under this condition, $u$ retains its neighbor(s) in $W_i$ which, in $L'$, are newly joined to $z_1$. This creates $|N_L(u) \cap W_i| \ge 1$ new triangles of type $s'_{11}$ as $uz_1 \in E(L')$. Therefore, the loss of $\{u, w_1, z_1\}$ is offset by these new adjacencies, and we have
\begin{align*}
    s'_{11}(u) + s'_{02}(u)
    \ge
     s_{11}(u) + s_{02}(u).
\end{align*}
In both instances, the inequality \eqref{eq:triangle-sum} holds.

\medskip
\noindent \textbf{Case 4:} $u = w_1$. \\
The construction of $L'$ involves the removal of the edge $u z_1$, which implies $s'_{11}(u) = 0$. However, this loss is more than compensated by the newly formed adjacencies between $u$ and $W_d$. Recall that $W_d \subseteq N_{L'}(u)$, and $|N_L(u) \cap W_d| \le 1$. Thus, every edge in the induced subgraph $L[W_d]$ forms a new triangle containing $u$ in $L'$, all of which are of type $s'_{02}$. It follows that $s'_{02}(u) \ge s_{02}(u) + e_d$.

In the  graph $L$, the triangles of type $s_{11}$ containing $u = w_1$ are formed by $z_1$ and its neighbors in $W_1 \setminus \{w_1\}$. Consequently, $s_{11}(u) = |N_L(u) \cap (W_1 \setminus \{w_1\})|$, which is bounded above by the total number of edges in $L[W_1]$, i.e., $s_{11}(u) \le e_1$. Using the condition $e_1 \le e_d$, we obtain
\begin{align*}
    s'_{11}(u) + s'_{02}(u)
    = s'_{02}(u) &\ge s_{02}(u) + e_d \\
    &\ge s_{02}(u) + e_1 \\
    &\ge s_{02}(u) + s_{11}(u).
\end{align*}
Thus, the inequality \eqref{eq:triangle-sum} holds for $u = w_1$. Since all cases for $v \in V(G)$ have been verified, $L'$ is a valid candidate for the extremal problem with $|E(L')| < |E(L)|$, a contradiction.
\end{proof}
This completes the proof of Lemma \ref{d1}.
\end{proof}

\subsection{Lemma \ref{id}}
\begingroup
\it
 Let $\theta = \frac{\sqrt{33}}{2\sqrt{928}}(t+1)$. For any vertex $u \in B_1$, let $k(u)=|\{D\in \mathcal{D}:N_G(u) \cap D \neq \emptyset\}|$.
    \begin{enumerate}
        \item[(i)] If $|N_G(u) \cap V(\mathcal{D})| \ge \sqrt{2\theta} + 3$, then $k(u) \le 3$.
        \item[(ii)] If $|N_G(u) \cap V(\mathcal{D})| \ge \sqrt{3\theta} + 5$, then $k(u) \le 2$.
\end{enumerate}
\endgroup

\begin{proof}
Suppose, for the sake of contradiction, that there exists a vertex $u \in B_1$ such that $k(u) \ge 4$ when $|N_G(u) \cap V(\mathcal{D})| \ge \sqrt{2\theta} + 3$, or $k(u) \ge 3$ when $|N_G(u) \cap V(\mathcal{D})| \ge \sqrt{3\theta} + 5$.

\medskip
Let $G'$ be the graph obtained from $G_d$ by deleting the edge $e = uv_{f_1(u)}$. By the extremality of $G_d$, there must exist a vertex $v \in V(G')$ that is contained in fewer than $\binom{t}{2}$ triangles in $G'$. Let $t_j(v)$ and $t_j'(v)$ denote the number of triangles containing $v$ that intersect $V(\mathcal{D})$ in exactly $j$ vertices in $G$ and $G'$, respectively. Clearly, $t_0'(v) = t_0(v)$.
Since   $e = uv_{f_1(u)}$ is the unique edge in $E(G_d) \setminus E(G')$, the  triangle degree of $v$ decreases only if $v \in \{u, v_{f_1(u)}\} \cup (N(u) \cap N(v_{f_1(u)}))$.

\medskip
We first observe that if $v \in V(\mathcal{D})$, then $v$ must belong to some $D\in \mathcal{D}$. By the construction of $G_d$ and subsequently
$G'$, each $D_i$ induces a $K_{t+1}$ in $G'$. Thus, any $v \in V(\mathcal{D})$ is contained in at least $\binom{t}{2}$ triangles in $G'$, which contradicts our choice of $v$. It follows that $v \in B_1\setminus V(\mathcal{D})$.  Let $\mu_i = |N_G(v) \cap D_i|$ for $1\le i\le d$ and $\mu = |N_G(v) \cap V(\mathcal{D})|$. We now distinguish the following two cases.
\medskip

\noindent
\textbf{Case 1:} $v \neq u$.
Its neighborhood in $V(\mathcal{D})$ is identical in both $G_d$ and $G'$. By the construction of $G_d$, the neighbors of $v$ are concentrated in $D_1$, satisfying $|N_{G'}(v) \cap D_1| =|N_{G_d}(v) \cap D_1| = \sum_{i=1}^d |N_G(v) \cap D_i|$. Expanding the binomial for $t_2'(v)$, we have
\begin{align}
    t_2'(v) &= \binom{\sum_{i=1}^d \mu_i}{2} \notag \\
    &= \sum_{i=1}^d \binom{\mu_i}{2} + \sum_{1 \le i < j \le d} \mu_i \mu_j \notag \\
    &= t_2(v) + \sum_{1 \le i < j \le d} \mu_i \mu_j. \label{eq:t1_decomp0}
\end{align}

We now estimate $t_1'(v)$. By isolating the contribution of $u$ in the neighborhood of $v$, we obtain
\begin{align}
    t_1'(v) &= |N_{G'}(u) \cap N_{G'}(v) \cap D_1| + \sum_{w \in (N_{G'}(v) \cap B_1) \setminus \{u\}} |N_{G'}(w) \cap N_{G'}(v) \cap D_1| \notag \\
    &= |N_{G_d}(u) \cap N_{G_d}(v) \cap (D_1 \setminus \{v_{f_1(u)}\})| + \sum_{w \in (N_{G_d}(v) \cap B_1) \setminus \{u\}} |N_{G_d}(w) \cap N_{G_d}(v) \cap D_1| \notag \\
    &\geq |N_{G_d}(u) \cap N_{G_d}(v) \cap (D_1 \setminus \{v_{f_1(u)}\})| + t_1(v) - |N_G(u) \cap N_G(v) \cap V(\mathcal{D})|. \label{eq:t1_decomp}
\end{align}
The inequality \eqref{eq:t1_decomp} follows from the fact that the construction of $G_d$ preserves or increases the size of common neighborhoods within $D_1$. By the definition of $G_d$, the first term in \eqref{eq:t1_decomp} is at least $|N_G(u) \cap N_G(v) \cap V(\mathcal{D})| - 1$.

In the case of equality, the nested property of neighborhoods in $G_d$ implies $N_G(u) \cap V(\mathcal{D}) \subseteq N_G(v) \cap V(\mathcal{D})$. Since $k(u) \ge 3$, this containment forces $v$ to intersect at least three cliques, ensuring that the term $\sum_{i<j} \mu_i \mu_j$  in \eqref{eq:t1_decomp0} is at least $1$. This additional contribution effectively offsets the unit loss in $t_1'(v)$. It follows that
\[
t'_2(v) + t'_1(v) + t'_0(v) \ge t_2(v) + t_1(v) + t_0(v) \ge \binom{t}{2},
\]
which contradicts the assumed extremality of $G$.

\medskip
\noindent \textbf{Case 2:} $v = u$. \\
Let $\mu = |N_G(u) \cap V(\mathcal{D})|$ and recall that $\eta := k(u)$ is the number of cliques intersected with the neighborhood of $u$. Without loss of generality, assume $N_G(u) \cap D_i \neq \emptyset$ for $1 \le i \le \eta$, and let $\mu_i = |N_G(u) \cap D_i|$.

We first evaluate $t_2'(u)$. In $G'$, the neighborhood of $u$ in $D_1$ has cardinality $\mu - 1$. A standard binomial expansion yields
\begin{align*}
    t_2'(u) &= \binom{\mu - 1}{2} = \binom{\mu}{2} - (\mu - 1) \\
    &= \sum_{i=1}^\eta \binom{\mu_i}{2} + \sum_{1 \le i < j \le \eta} \mu_i \mu_j - \mu + 1 \\
    &= t_2(u) + \sum_{1 \le i < j \le \eta} \mu_i \mu_j - \mu + 1.
\end{align*}

Next, we analyze $t_1'(u)$. Let $R = \{w \in N_G(u) \cap B_1 : N_G(u) \cap V(\mathcal{D}) \subseteq N_G(w) \cap V(\mathcal{D})\}$. Partitioning the neighbors of $u$ in $B_1$ according to their membership in $R$, we have
\begin{align}
    t_1'(u) &= \sum_{w \in R} |N_{G'}(w) \cap N_{G'}(u) \cap D_1| + \sum_{w \in (N_{G'}(u) \cap B_1) \setminus R} |N_{G'}(w) \cap N_{G'}(u) \cap D_1| \notag \\
%    &= \sum_{w \in R} |N_{G_d}(w) \cap N_{G_d}(u) \cap (D_1 \setminus \{v_{f_1(u)}\})| + \sum_{w \in (N_{G_d}(u) \cap B_1) \setminus R} |N_{G_d}(w) \cap N_{G_d}(u) \cap D_1| \notag \\
    &\ge \sum_{w \in R} (|N_G(w) \cap N_G(u) \cap V(\mathcal{D})| - 1) + \sum_{w \in (N_G(u) \cap B_1) \setminus R} |N_G(w) \cap N_G(u) \cap V(\mathcal{D})| \label{eq:t1_ineq} \\
    &= t_1(u) - |R|. \notag
\end{align}
%The inequality \eqref{eq:t1_ineq} arises from two observations: first, for $w \in R$, the nested neighborhood property in $G_d$ ensures that the removal of $v_{f_1(u)}$ reduces the common neighborhood size by exactly one; second, for $w \notin R$, the shifting operation in $G_d$ ensures that the common neighborhood size remains at least its original value in $G$, and its cardinality is unaffected by the removal of $v_{f_1(u)}$ due to the index ordering in $D_1$.
The inequality \eqref{eq:t1_ineq} follows from the fact that for any $w \in B_1$,
\begin{equation}\label{eq:w_ineq}
    |N_{G'}(w) \cap N_{G'}(u) \cap D_1| \ge |N_{G_d}(w) \cap N_{G_d}(u) \cap D_1| - 1 \ge |N_G(w) \cap N_G(u) \cap V(\mathcal{D})| - 1,
\end{equation}
where both inequalities in \eqref{eq:w_ineq} hold with equality if and only if $w\in R$.

Combining the previous estimates and noting $t_0'(u) = t_0(u)$, the total triangle degree for $u$ in $G'$ satisfies
\begin{equation*}
    \sum_{j=0}^2 t_j'(u) \ge \sum_{j=0}^2 t_j(u) + \sum_{1 \le i < j \le \eta} \mu_i \mu_j - \mu - |R| + 1.
\end{equation*}
By the extremality of $G$, we have $\sum t_j(u) \ge \binom{t}{2}$. Conversely, our assumption that $G'$ is not a valid candidate requires $\sum t_j'(u) < \binom{t}{2}$. This yields the following crucial lower bound for $|R|$:
\begin{equation} \label{eq:R_bound}
    |R| > \sum_{1 \le i < j \le \eta} \mu_i \mu_j - \mu + 1.
\end{equation}
The quadratic form $\sum_{i<j} \mu_i \mu_j$ subject to $\sum \mu_i = \mu$ and $\mu_i \ge 1$ is minimized when $\eta-1$ variables are set to $1$. Thus, we have the lower bound
\begin{equation*}
    \sum_{1 \le i < j \le \eta} \mu_i \mu_j \ge (\eta-1)(\mu - \eta + 1) + \binom{\eta-1}{2}.
\end{equation*}
Substituting this into \eqref{eq:R_bound}, a straightforward calculation shows that
\begin{equation} \label{eq:R_final}
    |R| > -\frac{1}{2}\eta^2 + \left(\mu + \frac{1}{2}\right)\eta - 2\mu + 1.
\end{equation}
By the definition of $R$, every vertex $w \in R$ satisfies $N_G(u) \cap V(\mathcal{D}) \subseteq N_G(w) \cap V(\mathcal{D})$, meaning each $w$ is adjacent to all $\mu$ neighbors of $u$ in $V(\mathcal{D})$. Consequently, the number of edges between $R$ and $V(\mathcal{D})$ is at least $\mu |R|$. Applying the bound from \eqref{eq:R_final}, we conclude that the number of edges between $B_1$ and $\bigcup_{i=1}^\eta D_i$ is at least
\begin{align*}
    \mu |R| >  \mu \left( -\frac{1}{2}\eta^2 + \left(\mu + \frac{1}{2}\right)\eta - 2\mu + 1 \right).
\end{align*}

By the definition of $\mathcal{A}_1$, the total excess degree of the $\eta$ cliques $D_1, \dots, D_\eta$ satisfies $\sum_{i=1}^\eta \sum_{x \in D_i} (d(x) - t) < \eta \theta$, where $\theta = \frac{\sqrt{33}}{2\sqrt{928}}(t+1)$. Since each edge between $B_1$ and $\bigcup_{i=1}^\eta D_i$ contributes to this sum, we have the inequality $\eta \theta > \mu |R|$. Substituting the lower bound for $|R|$ from \eqref{eq:R_final} and dividing by $\mu$, we obtain:
\begin{align}
    \frac{\eta \theta}{\mu} &> -\frac{1}{2}\eta^2 + \left(\mu + \frac{1}{2}\right)\eta - (2\mu - 1) \notag \\
    \implies \theta &> \left( \mu + \frac{1}{2} - \frac{1}{2}\left( \eta + \frac{4\mu - 2}{\eta} \right) \right) \mu. \label{eq:theta_master}
\end{align}
Let $g(\eta) = \eta + \frac{4\mu - 2}{\eta}$. Since $g(\eta)$ is convex, its maximum on any interval is attained at the boundaries.

First, suppose $\mu \ge \sqrt{3\theta} + 5$ and $\eta \ge 3$. Noting that $\eta \le \mu$, the right-hand side of \eqref{eq:theta_master} is minimized at the boundary $\eta = 3$. A simple calculation then yields:
\[
\theta > \left( \mu + \frac{1}{2} - \frac{1}{2}\left( 3 + \frac{4\mu - 2}{3} \right) \right) \mu = \frac{\mu^2 - 2\mu}{3}.
\]
For $\mu \ge \sqrt{3\theta} + 5$, it is easy to verify that $\frac{1}{3}(\mu^2 - 2\mu) > \theta$, which is a contradiction.

Similarly, if $\mu \ge \sqrt{2\theta} + 3$ and $\eta \ge 4$, we again observe that $\eta \le \mu$. The maximum of  the function $g(\eta) = \eta + \frac{4\mu - 2}{\eta}$  is attained at $\eta = \mu$. Substituting this boundary value into \eqref{eq:theta_master} gives:
\[
\theta > \left( \mu + \frac{1}{2} - \frac{1}{2}\left( \mu + \frac{4\mu - 2}{\mu} \right) \right) \mu = \frac{\mu^2 - 3\mu + 2}{2}.
\]
Given $\mu \ge \sqrt{2\theta} + 3$, the inequality $\frac{1}{2}(\mu^2 - 3\mu + 2) > \theta$ holds, yielding a contradiction and completing the proof of Lemma \ref{id}.
\end{proof}

\section{Conclusion}

In Theorem \ref{3}, we proved that every extremal graph for Problem \ref{13} with $k=3$ contains an isolated copy of $K_{t+1}$ whenever $n \geq ct^2 + o(t^2)$, where
$c = 1+\sqrt{928/33}$. It is natural to ask whether this structural property holds under
a linear bound on $n$
 in terms of $t$, which leads to the following conjecture.

\begin{conjecture}\label{conject}
There exists an absolute constant \(C\) such that, for every
integer \(t\ge2\) and every \(n\ge Ct\), every extremal graph for Problem~\ref{13}
with \(k=3\) contains an isolated copy of \(K_{t+1}\).
\end{conjecture}

One may also ask whether Conjecture \ref{conject} extends to non-integer $t$. The answer
is negative. Specifically, suppose $\lceil t \rceil$ is an even integer such that
$\lceil t \rceil + 2$ divides $n$. Let $G$ be the disjoint union of several copies of
$K_{\lceil t \rceil+2} \setminus M$, where $M$ is a perfect matching of size
$\frac{\lceil t \rceil+2}{2}$. In this construction, $G$ is a $\lceil t \rceil$-regular
graph where each vertex is contained in exactly $\binom{\lceil t \rceil}{2} -
\frac{\lceil t \rceil}{2}$ triangles. If $\binom{t}{2} \leq \binom{\lceil t \rceil}{2} -
\frac{\lceil t \rceil}{2}$, then $G$ serves as an extremal graph with no isolated copy of
$K_{\lceil t \rceil+1}$, providing a counterexample. The existence of this counterexample leads us to
the following problem.

\begin{problem}
For real $t \geq 2$, does there exist an absolute constant \(C\) such that,
whenever \(n\ge Ct\), every extremal
graph for Problem \ref{13} with $k=3$ contains  a connected component of size $\lceil t \rceil + 1$ or
$\lceil t \rceil + 2$?
\end{problem}

\section*{Acknowledgement}

The authors would like to express their sincere gratitude to Yi Zhao from Georgia State University for introducing this problem to us and for his insightful suggestions.


\begin{thebibliography}{99}



\bibitem{Ch}Z. Chase, The maximum number of triangles in a graph of given maximum degree, \newblock{\em Advances in Combinatorics (Online)} 10 (2020) 5pp.



\bibitem{Erdos} P. Erd\H{o}s, Some theorems on graphs, \newblock{\em Riveon Lematematika}, 9 (1955), 13--17.


\bibitem{Fra2}
P. Frankl, Erd\H{o}s-Ko-Rado theorem with conditions on the maximal degree, \newblock {\em J. Combin. Theory Ser. A} 46 (1987), 252--263.

\bibitem{Fra3}
P. Frankl, J. Han, H. Huang and Y. Zhao, A degree version of the Hilton-Milner theorem, \newblock {\em J. Combin. Theory Ser. A} 155 (2018), 493--502.


\bibitem{Fra4}
 P. Frankl and N. Tokushige, A note on Huang-Zhao theorem on intersecting families with large minimum degree, \newblock {\em Discrete Math.} 340 (2017), 1098--1103.


\bibitem{Fu2}
Z. F\"uredi and Y. Zhao, Shadows of 3-Uniform Hypergraphs under a Minimum Degree Condition, \newblock {\em SIAM J. Discrete Math.} 36 (2022), 2523--2533.

\bibitem{Han} H. H\`{a}n, Y. Person and M. Schacht, On perfect matchings in uniform hypergraphs with large minimum vertex degree, \newblock {\em SIAM J. Discrete Math.} 23 (2009), 732--748.


\bibitem{Huang1}
H. Huang and Y. Zhao, Degree versions of the Erd\H{o}s–Ko–Rado theorem and Erd\H{o}s hypergraph matching conjecture, \newblock {\em J. Combin. Theory Ser. A} 150 (2017), 233--247.

\bibitem{Huang2}
H. Huang and Y. Zhang, A $d$-degree generalization of the Erd\H{o}s–Ko–Rado theorem, \newblock {\em J. Combin. Theory Ser. A} 221 (2026) 106163.

\bibitem{liu}
H. Liu, M. Lu and Y. Zhang, Shadows of Uniform Hypergraphs under a Minimum Degree Condition, https://arxiv.org/abs/2605.02610.

\bibitem{Lovász}
L. Lovász, Combinatorial Problems and Exercises, 13.31 (North-Holland, Amsterdam, 1979).





\bibitem{Kha1} I. Khan, Perfect matching in 3-uniform hypergraphs with large vertex degree, \newblock {\em SIAM J. Discrete Math.} 27 (2013), 1021--1039.



\bibitem{Kuhn1} D. K\"{u}hn and D. Osthus, Matchings in hypergraphs of large minimum degree, \newblock {\em J. Graph Theory} 51 (2006), 269--280.



\bibitem{Kup}
A. Kupavskii, Degree versions of theorems on intersecting families via stability, {\em J. Combin. Theory Ser. A} 168 (2019), 272--287.

\bibitem{Lo}
A. Lo and K. Markstr\H{o}m, ${\ell}$-Degree Tur\'{a}n Density, \newblock {\em SIAM J. Discrete Math.} 28 (2014), 1214--1225.

\bibitem{Lo1}
A. Lo and Y. Zhao, Codegree Turán Density of Complete $r$-Uniform Hypergraphs, \newblock {\em SIAM J. Discrete Math.}, 32 (2018), 1154--1158.

\bibitem{Mar} K. Markstr\"{o}m and A. Ruci\'{n}ski, Perfect matchings (and Hamilton cycles) in hypergraphs with large degrees, \newblock {\em European J. Combin.} 32 (2011), 677--687.

\bibitem{Mu} D. Mubayi and Y. Zhao, Co-degree density of hypergraphs, \newblock {\em J. Combin. Theory Ser. A} 114 (2007), 1118--1132.

\bibitem{Pik} O. Pikhurko, Perfect matchings and $K^3_4$-tilings in hypergraphs of large codegree, \newblock {\em Graphs Combin.} 24 (2008), 391--404.



\bibitem{Rod2} V. R\"{o}dl, A. Ruci\'{n}ski and E. Szemer\'{e}di, Perfect matchings in uniform hypergraphs with large minimum degree, \newblock {\em European J. Combin.} 27 (2006), 1333--1349.


\bibitem{Rod3} V. R\"{o}dl, A. Ruci\'{n}ski and E. Szemer\'{e}di, Perfect matchings in large uniform hypergraphs with large minimum collective degree, \newblock {\em J. Combin. Theory Ser. A} 116 (2009), 613--636.

\bibitem{Si}
A. Sidorenko, Extremal Problems on the Hypercube and the Codegree Tur\'{a}n Density of Complete $r$-Graphs, \newblock {\em SIAM J. Discrete Math.} 32 (2018), 2667--2674.

\bibitem{TrZh12} A. Treglown and Y. Zhao, Exact minimum degree thresholds for perfect matchings in uniform hypergraphs,
\newblock {\em J. Combin. Theory Ser. A} 119 (2012), 1500--1522.

\bibitem{TrZh13}
A.~Treglown and Y. Zhao,
\newblock Exact minimum degree thresholds for perfect matchings in uniform
  hypergraphs {II},
\newblock {\em J. Combin. Theory Ser. A} 120 (2013), 1463--1482.



\bibitem{wang} Y. Wang and Y. Zhang, Vertex degree sums for perfect matchings in 3-uniform hypergraphs, \newblock {\em Discrete Math.} 348 (2025), 114564.

\bibitem{Yi4} Y. Zhang and M. Lu, Vertex degree sums for matchings in 3-uniform hypergraphs,  \newblock {\em Discrete Math.} 347 (2024), 113959.

\bibitem{zhang} Y. Zhang, Y. Zhao and M. Lu, Vertex degree sums for perfect matchings in 3-uniform hypergraphs,
\newblock {\em Electron. J. Combin.} 25 (2018), P3.45.

\bibitem{zhang2} Y. Zhang, Y. Zhao and M. Lu, Vertex degree sums for matchings in 3-uniform hypergraphs,
\newblock {\em Electron. J. Combin.} 26 (2019), P4.5.


\end{thebibliography}
\end{document}